\documentclass[11pt,reqno]{amsart}

\usepackage{amsmath,amsfonts,amssymb,mathrsfs,stackengine}
\usepackage{amssymb,mathrsfs,graphicx,extpfeil}
\usepackage{amsmath,amsfonts,amssymb,amscd,amsthm,bbm}
\usepackage{epsfig}

\usepackage{graphicx,colortbl,here}
\usepackage{extpfeil}
\usepackage{subcaption}
\usepackage{caption}
\usepackage{kotex}
\usepackage{hyperref}
 \usepackage {cite}
\newcommand{\bbr}{\mathbb R}

\newcommand*\di{\mathop{}\!\mathrm{d}}

\newcommand{\opnorm}[1]{{ \vert\kern-0.25ex \vert\kern-0.25ex \vert #1 
   \vert\kern-0.25ex \vert\kern-0.25ex \vert}}

\newtheorem{theorem}{Theorem}[section]
\newtheorem{lemma}{Lemma}[section]

\newtheorem{proposition}{Proposition}[section]
\newtheorem{remark}{Remark}[section]
\newtheorem{definition}{Definition}[section]

\hypersetup{
	hyperindex=true
}
\hypersetup{hidelinks}
\allowdisplaybreaks[4]
\begin{document}

\title[Stability and random mean-field limit for KCS model]{Weak stability and random mean-field limit of the phase-spatially extended  kinetic Cucker--Smale model} 

\author[Ha]{Seung-Yeal Ha}
\address[Seung-Yeal Ha]{\newline Department of Mathematical Sciences and Research Institute of Mathematics, \newline
	Seoul National University, Seoul, 08826, Republic of Korea}
\email{syha@snu.ac.kr}

\author[Wang]{Xinyu Wang$^{*}$}
\address[Xinyu Wang]{\newline Department of Mathematical Sciences, \newline
	Seoul National University, Seoul, 08826, Republic of Korea
\newline School of Mathematics \newline Harbin Institute of Technology, Harbin  150001, People's Republic of China}
\email{wangxinyu97@snu.ac.kr}

\thanks{\textbf{Acknowledgment.} 
	The work of S.-Y. Ha is supported by National Research Foundation (NRF) grant funded by the Korea government(MIST) (RS-2025-00514472), and the work of X. Wang is supported by the Natural Science Foundation of China (grants 123B2003), the China Postdoctoral Science Foundation (grants 2025M774290), and Heilongjiang Province Postdoctoral Funding (grants LBH-Z24167). $^{*}$Corresponding author.}

\begin{abstract}
We study the weak stability of the kinetic Cucker--Smale (in short, KCS) model in a phase-spatially extended setting, which can be formally derived from the infinite Cucker--Smale model in the mean-field limit. For a bounded Lipschitz communication weight function, we derive finite-time Osgood-type weak stability for measure-valued solutions with exponential velocity tails and finite spatial second moments. Unlike the phase-spatially confined setting, the solution operator to the KCS model is not Lipschitz continuous with respect to initial data. This is due to the unbounded velocity tail and the corresponding absence of a uniform Lipschitz bound for the alignment force. As an application of weak stability,  we obtain an i.i.d. sampling consequence: empirical measures generated from independent initial samples converge to the measure-value solution for the corresponding kinetic model in any finite time interval, in expectation. 
\end{abstract}



\maketitle


\section{Introduction}\label{sec:1}
\setcounter{equation}{0}

The dynamics of a large Cucker--Smale (in short, CS) ensemble \cite{Cucker,cucker2} can be effectively approximated by the corresponding mean-field kinetic model; see, for instance, \cite{kinetic4,k4,k2,kinetic2,ks2,kinetic5,w4,Bolley2011,k6,ks1}. The resulting kinetic equation belongs to the class of the Vlasov--McKean type models, namely, nonlinear transport equations on the phase space with a nonlocal forcing. In most of the literature, the supports of relevant physical observables are assumed to be confined in a bounded region of the phase space $\mathbb R^d_x\times \mathbb R^d_v,$
so that the kinetic density has a compact support in both position and velocity variables \cite{k4,kinetic4,a2}. Under the compact-support assumption, the well-posedness, stability, mean-field limit, and flocking dynamics of the kinetic Cucker--Smale (in short, KCS) model have been extensively studied; see \cite{k4,k2,k1,w4} and the references therein.

Recently, finite-in-time stability and mean-field limit in spatially extended settings have been considered in which the spatial support of the kinetic density is allowed to be unbounded, while the velocity support is still assumed to be bounded; see, for instance, \cite{w6,w2,HKPZ}. Thus, most previous stability and mean-field theories for the KCS model still rely on at least one compact support assumption, either compact phase-space support or compact velocity support. In this paper, we remove both restrictions and study stability and random mean-field approximation for the KCS model in a fully noncompact phase space (or phase-spatially extended setting). 

To fix the idea, we begin with the CS model. Let \(x_i(t)\) and \(v_i(t)\) denote the position and velocity of the \(i\)-th particle in \(\mathbb R^d\) at time $t$. Then the dynamics of $(x_i, v_i)$ is governed by the Cauchy problem for the  CS model with a system size $N$:
\begin{equation}\label{A-1}
	\begin{cases}
		\displaystyle \dot x_i=v_i,\quad t>0,\quad i \in [N] := \{1,\dots,N \}, \\[2mm]
		\displaystyle \dot v_i=\frac{\kappa}{N}\sum_{j=1}^N\phi(|x_i-x_j|)(v_j-v_i), \\[2mm]
		\displaystyle (x_i, v_i) \Big|_{t = 0+} = (x_{i,0}, v_{i,0}),
	\end{cases}
\end{equation}
where \(\kappa>0\) and $| \cdot |$ are the coupling strength and the standard $\ell^2$-norm in the Euclidean space $\bbr^d$, $\bbr_+ := [0, \infty)$, respectively, and \(\phi\) is a nonnegative communication weight satisfying Lipschitz continuity, positivity, and uniform boundedness:
\[
\phi\in W^{1,\infty}(\bbr_+),\quad 0\le \phi(r) \le \|\phi\|_{L^\infty}, \quad r \geq 0.
\]
In the mean-field limit $(N \to \infty)$, the corresponding Cauchy problem for the KCS  equation reads as 
\begin{equation} 
	\begin{cases} \label{A-2}
		\displaystyle \partial_tf+ \nabla_x \cdot (v f) + \nabla_v \cdot (L[f]f)=0, \quad t > 0,~~(x, v) \in {\mathbb R}^{2d},  \vspace{6pt}\\
		\displaystyle L[f](t,x,v)=-\kappa \int_{\mathbb{R}^{2d}} \phi( |x-x_*|) \left(v-v_*\right)  f(t, x_*, v_*) \di x_* \di v_*, \\
		\displaystyle f \Big|_{t = 0}  =f_0,
	\end{cases}
\end{equation}
where $f_0$ is the initial datum satisfying positivity, boundedness, and integrability conditions:
\begin{equation} \label{A-3}
	f_0 \in(L^1 \cap L_+^\infty)(\mathbb{R}^{2d}) \quad \mbox{and} \quad  ( |x|^2 + |v|^2)f_0 \in L^1(\mathbb{R}^{2d}).
\end{equation}
Note that the global existence of weak solutions to \eqref{A-2} has already been studied in \cite{kinetic3} without uniqueness. In \cite{HWKCS2026}, the authors established the uniqueness of weak solution  under the additional requirement that initial kinetic density decays exponentially fast in the velocity variable, i.e., there exists a positive constant $\alpha>0$ such that
\begin{align*}
	\int_{\mathbb{R}^{2d}} e^{\alpha |v|} {f_0(x, v) \di x \di v}<\infty.\end{align*}
However, the stability results of the KCS model \eqref{A-2} in a fully phase-spatially extended setting are not known yet. For compactly supported initial data, the standard Dobrushin argument yields finite-in-time Lipschitz stability in Wasserstein distance \cite{kinetic4,k4,k2,k1}. Indeed, if the position and velocity supports remain to be bounded on a finite time interval, then the force field \(L[f]\) is Lipschitz on the relevant part of phase space, and usual Gr\"onwall type argument can be applied to derive a finite-time stability estimate \cite{w6,HWGKJK2026,HWRKCS2025,NP}. However, the situation is different, when the velocity support of kinetic density is unbounded. Although the communication weight is bounded and Lipschitz, the \(x\)-Lipschitz bound of the force contains a factor comparable to \(|v|+|v_*|\). Thus, no uniform Lipschitz constant is available on the full phase space. For notational simplicity, we use the abbreviated notation:
\[ z = (x,  v), \quad   \di z = \di x \di v.  \]
Before we move on further, we recall the concept of ``{\it weak solution}" to the Cauchy problem \eqref{A-2}-\eqref{A-3} and weak stability as follows. 
\begin{definition}\label{D1.1}
\emph{(Weak solution)}
For $T  \in (0,\infty]$, let $f\in {\mathcal C}([0, T); L_+^1(\mathbb{R}^{2d}))$ be a weak solution to  \eqref{A-2}  if the following relations hold:
	\begin{enumerate}
		\item
		$f$ is weakly continuous in $t$:~$\forall~\psi \in {\mathcal C}_c^1(\mathbb{R}^{2d})$, 
		\[
		t \quad \mapsto \quad \int_{\mathbb{R}^{2d}}\psi f \di z \quad\text{is continuous}.
		\]
		\item	
		$f$ satisfies \eqref{A-2} in weak sense:~$ \forall~\zeta \in {\mathcal C}_c^1([0, T) \times\mathbb{R}^{2d})$ and $\tau \in (0, T)$, 
\[ \int_{\mathbb{R}^{2d}}\zeta(t, z) f(t, z) \di z =  \int_{\mathbb{R}^{2d}}\zeta(0,z) f_0(z) \di z + \int_0^{\tau} \int_{\mathbb{R}^{2d}}(\partial_s\zeta+v\cdot \nabla_x \zeta +\nabla_v\zeta \cdot L[f])f(s, z) \di z \di s.\]
	\end{enumerate}	
\end{definition}
For given weak solutions $f$ and $g$, we introduce measure-valued solutions to \eqref{A-2} whose Radon-Nikodym derivatives are $f$ and $g$:
\[ \di \mu_t(z) = f(t, z) \di z, \quad  \di \nu_t(z) = g(t, z) \di z. \]
Next, we introduce a concept of weak stability of measure-valued solutions to \eqref{A-2} as follows.
\begin{definition} \label{D1.2}
\emph{(Weak stability \cite{Te})}
For $T \in (0, \infty]$, let \(\mu_t\) and \(\nu_t\) be two measure-valued solutions to \eqref{A-2} on \([0,T) \). The Cauchy problem for the KCS model \eqref{A-2} satisfies a finite(-in)-time weak stability in 1-Wasserstein metric if there exists a continuous and nonnegative function $(t, \zeta) \mapsto G_t(\zeta) = G(t, \zeta)$ such that 
\begin{equation} \label{A-4}
W_1(\mu_t, \nu_t) \leq G_T\Big(W_1(\mu_0, \nu_0) \Big), \quad t \in (0, T),
 \end{equation}
where $G_t(\zeta) \to 0$~~as $\zeta \to 0$ for fixed $t > 0$ and $t \mapsto G_{t}(\zeta)$ is increasing for a fixed $\zeta$. 
\end{definition}
\begin{remark} If the gauge function $G_t(\cdot)$ is independent of $t$ and it is not a linear function, then we call the \eqref{A-4} as the uniform(-in)-time weak stability estimate. In contrast, if the map $\zeta \mapsto G_T(\zeta)$ is linear, then we call \eqref{A-4} as the finite-time (strong) stability. Similarly, we can define uniform-time (strong) stability.  
\end{remark}
The weak and strong stabilities \eqref{A-4} have been studied in diverse disciplines, e.g., hyperbolic conservation, kinetic equations, and collective dynamics. More precisely, for the one-dimensional hyperbolic conservation laws, weak stability has been first studied by B. Temple \cite{Te} for the class of Glimm's weak entropy solutions. Later, Temple's weak stability has been further extended to $L^1$-stability in \cite{B-L-Y, L-Y}. Similarly, uniform-time weak stability in $L^1$-norm has been established for the Boltzmann equation by X. Lu \cite{Lu}, and Lu's result was also extended to the uniform $L^1$-stability near vacuum by Ha \cite{Ha}. Ha and his collaborators \cite{kinetic2} also investigated the uniform stability of the KCS models \eqref{A-2} in the context of phase-spatially confined setting: for any two measure-valued solutions $\mu_t$ and $\nu_t$ with compact supports, 
\begin{equation}\label{A-5}
\sup_{0 \leq t < \infty} W_1(\mu_t, \nu_t) \leq G W_1(\mu_0, \nu_0),
\end{equation}
for a positive constant $G$ independent of $t$. For phase-spatially extended setting, the authors in reference \cite{HWY} explicitly constructed a counterexample for the breakdown of uniform-time strong stability of KCS model \eqref{A-2} in 1-Wasserstein distance. The same counterexample can also be applied for the breakdown of uniform weak stability as well (see  Remark \ref{R3.1}). Thus, a natural question is as follows.
\vspace{0.1cm}
\begin{quote}
``Does the Cauchy problem \eqref{A-2} for the KCS model satisfy finite-time weak stability or not ?"
\end{quote}
\vspace{0.1cm}
In this paper, we address this question in an affirmative manner. In fact, our main results of this paper can be summarized as follows. \newline 

First, we derive the finite-time weak-stability \eqref{A-4} in $1$-Wasserstein metric in the sense of Definition \ref{D1.2}. More precisely, for \(T>0\),~\(M>0\),~\(\alpha>0\),  assume that initial data $f_0$ and $g_0$ satisfy the boundedness moment condition:
\begin{equation*}
	\max\left\{\int_{\mathbb R^{2d}}\bigl(|x|^2+e^{\alpha |v|}\bigr) f_0(z) \di z ,\quad \int_{\mathbb R^{2d}}\bigl(|x|^2+e^{\alpha |v|}\bigr) g_0(z) \di z \right\}\le M,
\end{equation*} and let $\mu_t$ and $\nu_t$ be two measure-valued solutions to \eqref{A-2} in the time interval \([0,T)\) with densities $f, g \in {\mathcal C}([0, T); L_+^1(\mathbb{R}^{2d})) \cap L_+^\infty([0, T) \times \mathbb{R}^{2d})$. Then, there exists a continuous and concave gauge function $G_T$ such that 
\begin{equation} \label{A-6}
\sup_{0\le t \le T}W_1(\mu_t,\nu_t) \le G_T \Big(W_1(\mu_0, \nu_0) \Big).
\end{equation}
For more details, we refer to Theorem \ref{T3.1}.\newline

Second, as a direct application of finite-time weak stability \eqref{A-6}, we derive a random mean-field limit in any finite-time interval $[0, T)$. For well-prepared initial datum $f_0$ and random initial configuration $Z_0 = \{ (x_{i,0}, v_{i,0}) \}$ with the common law $f_0$, let  $f$ and \(\mu^N_t\)  be the weak solution to \eqref{A-2} and the random empirical measure generated by \eqref{A-1} with initial configuration $Z_0$. Then,  for every finite $T \in (0, \infty)$, there exists positive constants $C_T$ and $\eta_T$ such that
\begin{equation} \label{A-7}
		\mathbb E\left[\sup_{0\le t\le T}W_1(\mu^N_t, \mu_t) \right]\le \frac{C_T}{N^{\eta_T}}.
\end{equation}
Note that as $N \to \infty$, the right-hand side of \eqref{A-7} tends to zero. For details on the fluctuation estimate \eqref{A-7}, we refer to Theorem \ref{T3.2}. \newline

The rest of this paper is organized as follows. In Section \ref{sec:2}, we study the basic properties of the particle and kinetic CS models \eqref{A-1} and \eqref{A-2}, and recall previous results on emergent dynamics and stability results. In Section \ref{sec:3}, we briefly summarize two main results on the finite-time weak stability and random mean-field limit in any finite-time interval with several comments. In Section \ref{sec:4}, we provide the detailed proof of finite-time weak stability of measure-valued solutions to \eqref{A-2}. In Section \ref{sec:5}, we present the proof of random mean-field limit in any finite-time interval, as a direct application of weak stability. Finally, Section \ref{sec:6} is devoted to a brief summary of main results and some remaining issues for  future work.

 \vspace{0.5cm} 
 
\noindent\textbf{Notation:} 
For $f$ and $g$, we say that $f \lesssim g$ if and only if there exists a positive constant $C$ such that $f \leq C g$, and throughout the paper, we use abbreviated notation:
\[ z = (x,  v), \quad  z_* = (x_*,  v_*), \quad   \di z = \di x \di v, \quad  \di z_* = \di x_* \di v_*.  \]

 \vspace{0.5cm} 
%
%
%

\section{Preliminaries}\label{sec:2}
\setcounter{equation}{0}
In this section, we first recall several basic estimates of the CS model, which will be crucially used in later sections. Second, we recall previous results on the measure-valued solution for the KCS model in phase-spatially confined and extended settings. 

\subsection{Particle and kinetic CS models} \label{sec:2.1}
In this subsection, we present elementary estimates for the particle CS model \eqref{A-1}, basic measure-theoretic formulation and forward particle trajectory associated with KCS model \eqref{A-2}.
\subsubsection{The CS model} \label{sec:2.1.1}
In this subsection, we study basic estimates for \eqref{A-1} regarding the conservation and energy dissipation. 
\begin{lemma} \label{L2.1}
For a constant $T \in (0, \infty]$, let $\{ (x_i, v_i) \}$ be a solution to \eqref{A-1}. Then, for $t \in [0, T)$, the following estimates hold.
\begin{align*}
\begin{aligned}
\noindent & (i)~\frac{\di }{\di t} \sum_{i=1}^{N} v_i(t) = 0, \quad \frac{\di }{\di t} \sum_{i=1}^{N} |v_i(t)|^2 =  -\frac{\kappa}{N} \sum_{i, j=1}^N  \phi(|x_i(t)-x_j(t)|) |v_j(t) - v_i(t)|^2 \leq 0.  \\
\noindent & (ii)~\frac{1}{N} \sum_{i=1}^N(|x_i(t)|+|v_i(t)|) \leq (1 + T)
\left[
\frac1N\sum_{i=1}^N|x_{i,0}|
+ \left(\frac1N\sum_{i=1}^N |v_{i,0}|^2\right)^{1/2}
\right].
\end{aligned}
\end{align*}
\end{lemma}
\begin{proof} 
(i) For the conservation of total momentum, we sum $\eqref{A-1}_2$ over all $i \in [N]$ and use the index exchange transformation $(i, j)~\longleftrightarrow~(j, i)$ to see 
\begin{align*}
\begin{aligned}
\frac{\di }{\di t} \sum_{i=1}^{N} v_i &=\frac{\kappa}{N}\sum_{i,j=1}^N\phi(|x_i-x_j|)(v_j-v_i)  \\
&= -\frac{\kappa}{N}\sum_{i,j=1}^N\phi(|x_i-x_j|)(v_j-v_i) = -\frac{\di }{\di t} \sum_{i=1}^{N} v_i.
\end{aligned}
\end{align*}
This yields
\[ 
\frac{\di }{\di t} \sum_{i=1}^{N} v_i= 0.
\]
For an energy dissipation estimate, we take an inner product $2v_i$ with $\eqref{A-1}_2$, sum the resulting relation over all $i \in [N]$  and use the index exchange transformation again to find 
\begin{align*}
\begin{aligned}
\frac{\di }{\di t} \sum_{i=1}^{N} |v_i|^2  &= \frac{2\kappa}{N} \sum_{i, j=1}^N \phi(|x_i-x_j|) v_i \cdot (v_j-v_i)  \\
&= -\frac{2\kappa}{N} \sum_{i, j=1}^N   \phi(|x_i-x_j|)  v_j \cdot (v_j - v_i) \\
& = -\frac{\kappa}{N} \sum_{i, j=1}^N  \phi(|x_i-x_j|) |v_j - v_i|^2 \leq 0,
\end{aligned}
\end{align*}
where the third relation follows from the addition of the first and the second relations and then divide the resulting relation by 2. \newline

\noindent (ii)~By the result of (i), one has 
\begin{equation}\label{B-0-0}
	\frac1N\sum_{i=1}^N |v_i(t)|^2
	\le
	\frac1N\sum_{i=1}^N |v_{i,0}|^2,
	\quad 0\le t\le T.
\end{equation}
We use $\eqref{A-1}_1$ to see
\begin{equation} \label{B-0-1}
|x_i(t)|
\le
|x_{i, 0}|+\int_0^t |v_i(s)| \di s.
\end{equation}
Now, we use the Cauchy--Schwarz inequality, \eqref{B-0-0} and \eqref{B-0-1} to get 
\begin{equation} \label{B-0-2}
\frac1N\sum_{i=1}^N |v_i(t)|
\le
\left(\frac1N\sum_{i=1}^N |v_i(t)|^2\right)^{1/2}
\le
\left(\frac1N\sum_{i=1}^N |v_{i,0}|^2\right)^{1/2},
\end{equation}
and
\begin{align}
\begin{aligned} \label{B-0-3}
	\frac1N\sum_{i=1}^N |x_i(t)|
	&\le
	\frac1N\sum_{i=1}^N |x_{i,0}|
	+
	\int_0^t
	\frac1N\sum_{i=1}^N |v_i(s)| \di s \\
	&\le
	\frac1N\sum_{i=1}^N |x_{i,0}|
	+
	T\left(\frac1N\sum_{i=1}^N |v_{i,0}|^2\right)^{1/2}, \quad t \in [0, T).
\end{aligned}
\end{align}
Finally, we add \eqref{B-0-2} and \eqref{B-0-3} to get the desired estimate. 
\end{proof}

\subsubsection{Measure-theoretic formulation} \label{sec:2.1.2}
Let $\mathcal{P}(\bbr^{2d})$ be the set of all probability measures on the phase space $\bbr^{2d}$, which can be understood as normalized nonnegative bounded linear functional on ${\mathcal C}_0(\bbr^{2d})$, and we also introduce a subspace ${\mathcal P}_p(\bbr^{2d})$ with finite $p$-th moments:~for $p \in [1, \infty)$, 
\[ {\mathcal P}_p(\bbr^{2d}) := \left\{ \mu \in {\mathcal P}(\bbr^{2d}):~\int_{\bbr^{2d}} (|x|^p + |v|^p) \di \mu(x,v) < \infty \right\}.  \]
For a probability measure $\mu\in\mathcal{P}(\bbr^{2d})$, we use a standard duality relation between measure and test function:
\[\langle\mu, \varphi \rangle :=\int_{\bbr^{2d}} \varphi(x,v) \di \mu(x,v),\quad \varphi \in {\mathcal C}_0(\bbr^{2d}).\]
Next, we introduce a metric in the space of ${\mathcal P}_p(\bbr^{2d})$. 
	\begin{definition}\label{D2.1}
		\emph{(Support of measure and $p$-Wasserstein distance \cite{k8})} 
For $p \in [1, \infty)$, let $\mu, \nu \in \mathcal{P}_p(\mathbb{R}^{2d})$.  
\begin{enumerate}
\item
\mbox{spt}($\mu$) (the support of a measure $\mu$) is the closure of the set
consisting of all points $z$ in $\bbr^{2d}$ such that 
\[ \mu(B_r(z)) > 0, \quad \forall~r > 0. \]
\item		
The Wasserstein distance of order $p$ between $\mu$ and $\nu$ (or $p$-Wasserstein distance between $\mu$ and $\nu$) is defined as follows:
	\begin{equation*}
		W^p_p(\mu, \nu):=\inf\left\{\int_{\mathbb{R}^{2d}} |z-z_*|^p \di \pi(z, z_*):\pi\in\Pi(\mu, \nu)\right\},
	\end{equation*}
	where $\Pi(\mu,\nu)$ is the set of all probability measures on $\mathbb{R}^{2d}$ with marginals $\mu$ and $\nu$ in ${\mathcal P}(\bbr^d)$, respectively. 	We call any element $\pi$ in $\Pi(\mu, \nu)$ as the coupling of $\mu$ and $\nu$. 
\end{enumerate}	
\end{definition}
Now, we are ready to recall the concept of measure-valued solution to \eqref{A-2} which is parallel to the concept of weak solution in Definition \ref{D1.1}.
\begin{definition} \label{D2.2}
\emph{ (Measure-valued solution \cite{k2})} For $T\in[0,\infty)$, let $\mu \in L^{\infty}([0,T);\mathcal{P}(\bbr^{2d}))$ be a measure-valued solution to \eqref{A-2} with the initial measure $\mu_0 \in\mathcal{P}(\bbr^{2d})$ if and only if the following three assertions hold:
\begin{enumerate}
\item Total mass is normalized: $\langle\mu_t,1\rangle=1$.
\item $\mu$ is weakly continuous in $t$:
\[ t \quad \mapsto \quad \langle\mu_t, \varphi \rangle~\mbox{is continuous} \quad \forall~\varphi \in {\mathcal C}_0^1(\bbr^{2d}). \]
\item $\mu$ satisfies \eqref{A-2} in weak sense: for any $\psi \in {\mathcal C}^1_0([0, T) \times \bbr^{2d})$ and $t \in (0, T)$, 
\begin{equation}\label{msol}
 \langle\mu_t, \psi(t,\cdot,\cdot)\rangle-\langle \mu_0, \psi(0,\cdot,\cdot)\rangle=
\int_0^t \Big \langle\mu_s,\partial_s \psi + v\cdot\nabla_x \psi + L[f]\cdot\nabla_v \psi  \Big \rangle \di s.
\end{equation}
\end{enumerate}
\end{definition}
\begin{remark} \label{R2.1} Below, we provide two comments to be used in later sections.
\begin{enumerate}
\item
For a weak solution $f$ to \eqref{A-2}, if we introduce a measure $\mu$ by  
\[ \di \mu_t(x, v) = f(t,x,v) \di x \di v. \]
Then, $\mu$ is also the measure-valued solution to \eqref{A-2}. 
\vspace{0.1cm}
\item
Consider the empirical measure $\mu^N$ generated by the solution to \eqref{A-1}:
\[
\mu^N_t:=\frac1N\sum_{i=1}^N\delta_{(x_i(t),v_i(t))}.
\]
Then it is a measure-valued solution to \eqref{A-2}. For any
\(\psi \in C_c^1([0, T) \times \mathbb R^{2d})\), we use \eqref{A-1}:
\[
\begin{cases}
\displaystyle \dot x_i = v_i, \quad t > 0, \vspace{6pt}\\
\displaystyle \dot v_i =
L[\mu^N_t](x_i,v_i),
\end{cases}
\]
 to see
\begin{align}
\begin{aligned} \label{B-0}
& \frac{\di}{\di t}\int_{\mathbb R^{2d}} \psi(t, x,v) \di \mu^N_t = \frac{1}{N} \sum_{i=1}^{N}  \frac{\di}{\di t} \psi(t, x_i,v_i)  \\
& \hspace{0.5cm} = 
\frac1N\sum_{i=1}^N
\left[ \partial_t \psi(t,x_i, v_i) + 
\nabla_x\psi(t,x_i,v_i)\cdot {\dot x}_i
+
\nabla_v\psi(t,x_i,v_i)\cdot \dot v_i
\right] \\
& \hspace{0.5cm}  = \frac1N\sum_{i=1}^N
\left[ \partial_t \psi(t,x_i, v_i) + 
\nabla_x\psi(t,x_i,v_i)\cdot v_i
+
\nabla_v\psi(t,x_i,v_i)\cdot L[\mu^N](x_i,v_i)
\right].
\end{aligned}
\end{align}
We integrate \eqref{B-0} from $0$ to $t$ to derive the defining relation \eqref{msol}.  Thus \(\mu_t^N\) satisfies the weak formulation of the Cauchy problem \eqref{A-2}. \end{enumerate}
\end{remark}
\vspace{0.2cm}

\subsubsection{Characteristic flow } \label{sec:2.1.3}
For a given $z = (x, v) \in {\mathbb R}^{2d}$ and $t > 0$, we define a forward bi-characteristics $ (X(t), V(t)) : = (X(t, 0,z), V(t,0,z))$ issued from $z$ at time $t = 0$ as the unique solution for the following Cauchy problem:
\begin{equation} \label{B-1}
	\begin{cases}
		\displaystyle \dot X(t)=V(t), \quad t > 0,  \vspace{4pt} \\
		\displaystyle \dot V(t)=L[f](t,X(t),V(t)), \vspace{4pt} \\
		\displaystyle (X(0), V(0)) = z,
	\end{cases}
\end{equation}
where the velocity alignment force along the particle trajectory is given by the following relation:
\begin{equation} \label{B-2}
 L[f](t,X(t),V(t))=-\kappa\int_{\mathbb{R}^{2d}} \phi( |X(t) -x_*|) \left(V(t)-v_*\right)  f(t, z_*) \di z_*.
\end{equation}
Note that $L[f]$ in \eqref{B-2} is continuous in $t$ and Lipschitz continuous in state variables $(X, V)$, thus if we have weak solution $f$ to \eqref{A-2}, the standard Cauchy-Lipschitz theory  provides a local well-posedness of  particle trajectories near $t = 0$. Moreover,
the vector field generated by the right-hand side of \eqref{B-1} is locally bounded in $X$ and sub-linear in $V$, hence particle trajectory is globally well-defined. Thus the particle trajectory map $( X(t,0,z), V(t,0,z))$ is a well-defined homeomorphism for each fixed time $t$ and a ${\mathcal C}^1$-function of time $t$.
We set 
\[ \Phi^t(z) = (X(t,0, z), V(t,0, z)). \]
The weak solution $\mu_t = f(t,z) \di z$ to \eqref{A-2} is also given by the push-forward of the initial measure $\di \mu_0(z) = f_0(z) \di z$:
\[ \mu (t)= \Phi^t_\#\mu_0. \]
This is equivalent to 
\begin{equation} \label{B-2-0}
	\int_{\mathbb{R}^{2d}}\varphi(z) f(t, z) \di z =\int_{\mathbb{R}^{2d}}\varphi(X(t,0,z),V(t,0,z))f_0(z) \di z, \quad \forall~\varphi \in {\mathcal C}_b^0(\mathbb{R}^{2d}).
\end{equation}
As a direct application of \eqref{B-2-0}, we have the following elementary estimates. 
\begin{lemma} \label{L2.2}
\emph{\cite{HWKCS2026}}
Let $f = f(t,z)$ be a weak solution to \eqref{A-2} with sufficiently fast decay at infinite in phase space, and let $(X(t), V(t))$ be an associated particle trajectory defined in \eqref{B-1}. Then, the following estimates hold. 
\begin{align*}
\begin{aligned}
& (i)~  \int_{\mathbb{R}^{2d}}  f(t,z) \di z = \int_{\mathbb{R}^{2d}}  f_0(z) \di z, \quad \int_{\mathbb{R}^{2d}}  v f(t,z) \di z = \int_{\mathbb{R}^{2d}} v f_0(z) \di z. \\
& (ii)~  \int_{\mathbb{R}^{2d}} v f(t,z) \di z = \int_{\mathbb{R}^{2d}} V(t) f_0(z) \di z. \\
& (iii)~  \int_{\mathbb{R}^{2d}} |v|^D f(t,z) \di z = \int_{\mathbb{R}^{2d}} |V(t)|^D f_0(z) \di z  \quad \mbox{for}~D \geq 1. \\
& (iv)~ \int_{\mathbb{R}^{2d}} | v- v_c(t) |^2f(t,z) \di z = \int_{\mathbb{R}^{2d}} |V(t)- v_{c}(0)|^2f_0(z) \di z.
\end{aligned}
\end{align*}
\end{lemma}
\begin{proof}
For details, we refer to [Lemma 2.2, \cite{HWKCS2026}].
\end{proof}
\subsection{Previous results} \label{sec:2.2}
In this subsection, we discuss previous results on the well-posedness and large-time behaviors of \eqref{A-2} in two different settings: 
\vspace{0.1cm}
\begin{quote}
`` Spatially confined and extended settings with compact velocity support \\
$ \mbox{v.s.} \quad$ Phase-spatially extended setting. "
\end{quote}
\vspace{0.1cm}
Note that the latter case corresponds to the case with unbounded supports in space and velocity. 
\begin{proposition} 
\emph{(Compact velocity support \cite{k4,k1,w2,w6}) } \label{P2.1}
For $T \in (0, \infty]$, we assume that the initial data $f_0,g_0\in (L^1 \cap L^{\infty}_+)(\mathbb{R}^{2d})$ have compact velocity support, i.e., there exists a positive constant $P_\infty$ such that 
\begin{equation*}
	P_\infty :=\max\left\{\sup\limits_{(x,v)\in{\rm spt}(f_0)} |v|,~~\sup\limits_{(x,v)\in{\rm spt}(g_0)} |v|\right\}<\infty.
\end{equation*}
Then, there exist unique weak solutions $f,g\in L_+^{\infty}([0, T); L^1(\mathbb{R}^{2d}))$ to \eqref{A-2} with the initial data $f_0$ and $g_0$, respectively, with the following properties:
\vspace{0.1cm}
\begin{enumerate}
	\item 
	(Propagation of compact velocity supports): 
	\begin{equation}\label{B-3}
		\sup\limits_{(x,v)\in{\rm spt}(f(t))} |v|\le P_\infty ,  \quad  \sup\limits_{(x,v)\in{\rm spt}(g(t))} |v|\le P_\infty, \quad \forall~ t \in [0, T).
	\end{equation} 
	\item	
	(Conservation of total mass and  momentum):
	\[ \frac{\di}{\di t}\int_{\mathbb{R}^{2d}}f(t, z) \di z=0, \quad  \frac{\di}{\di t}\int_{\mathbb{R}^{2d}}vf(t, z) \di z =0, \quad \forall~ t \in [0, T). 
	\]
	\item
	(Finite-time stability in $W_p$):~ there exists a positive constant $G_T = G(T, P_\infty, \beta, p)$ such that
	\begin{equation*}
	 \sup_{0 \leq t < T} W_p(f(t),g(t))\le G_T W_p(f_0,g_0),
	\end{equation*}
	where $\beta$ is the spatial decay of  Cucker--Smale's communication weight $\phi=1/(1+s)^{\beta}$.
	\vspace{0.1cm}
\end{enumerate}	
\end{proposition}
Next,  we state the existence, uniqueness, and long-time behavior of a weak solution to \eqref{A-2} in a fully phase-spatially extended setting.
\begin{proposition}  \label{P2.2}
	\emph{(Phase-spatially extended setting \cite{kinetic3,HWKCS2026})} 
	Suppose that the initial probability density function \( f_0 \) satisfies
	\[
	f_0 \in (L^1 \cap L^{\infty}_+)(\mathbb{R}^{2d}), \quad |z|^2 f_0 \in L^1(\mathbb{R}^{2d}).
	\]
	Then the following assertions hold.
	\begin{enumerate}
		\item
		For $T \in (0, \infty)$, there exists a weak solution $f \in {\mathcal C}([0, T) ; L^1(\mathbb{R}^{2d})) \cap L_+^\infty([0, T) \times \mathbb{R}^{2d})$  to \eqref{A-2} such that
	\[
	|z|^2 f(t) \in L^1(\mathbb{R}^{2d}), \quad 	\int_{\mathbb{R}^{2d}}f(t, z) \di z =\int_{\mathbb{R}^{2d}}f_0(z) \di z, \quad t \in (0, T).
	\]
	If we assume that \( f_0 \) satisfies 
	\[ e^{\alpha|v|}f_0 \in L^1(\mathbb{R}^{2d}), \quad \alpha>0, \]
	 then, a weak solution to \eqref{A-2}  is unique.
	\vspace{0.2cm}
	\item If the initial density $f_0$ satisfies one of the following decay conditions: 
 \begin{align*}
 \begin{aligned}
&  \mbox{Either}~~{\mathcal M}_p(f_0, D_1, D_2) :=  \int_{\mathbb{R}^{2d}} (|x|^{D_1}+ |v|^{D_2}) f_0(z) \di z < \infty, \\
& \mbox{or}~~ {\mathcal M}_e(f_0, \alpha, \delta) := \int_{\mathbb{R}^{2d}} e^{\alpha \left(| x|+| v |^{\delta}\right)} f_0(z) \di z < \infty,
\end{aligned}
\end{align*}
	then the weak flocking emerges asymptotically under suitable conditions on parameters:
	
	 \begin{equation*}
		\begin{cases}
			\displaystyle \sup\limits_{0\le t<\infty}\int_{\mathbb{R}^{2d}} |x- x_c(0) - t v_c(0) |^2f(t, z) \di z < \infty, \vspace{6pt}\\
			\displaystyle  \lim_{t \to \infty} \int_{\mathbb{R}^{2d}} |v- v_c(0) |^2f(t, z)  \di z  = 0. 
		\end{cases}
	\end{equation*}
Here, $x_c(0)$ and $v_c(0)$ are the spatial and velocity averages of initial datum $f_0$:
\[ x_c(0) :=  \frac{\int_{\bbr^{2d}} x f_0(z) \di z}{\int_{\bbr^{2d}}  f_0(z) \di z}, \quad  v_c(0) :=  \frac{\int_{\bbr^{2d}} v f_0(z) \di z}{\int_{\bbr^{2d}}  f_0(z) \di z}.
\]	
	
	\end{enumerate}
\end{proposition}
\vspace{0.2cm}
In particular, we have the following spatial and velocity moment estimates. 
\begin{lemma} \label{L2.3}
\emph{ \cite{HWKCS2026}}
	For $T \in (0, \infty], \alpha >0,~\delta \geq 1$, let $f \in {\mathcal C}([0, T) ; L^1(\mathbb{R}^{2d})) \cap L_+^\infty([0, T) \times \mathbb{R}^{2d})$ be a weak solution to \eqref{A-2} such that ${\mathcal M}_{e}(f_0, \alpha, \delta) < \infty$. Then, the following assertions hold.
\begin{enumerate}
\item	
The propagation of velocity moments hold.
	\begin{equation} \label{B-4}
		\begin{cases}
			\displaystyle \int_{\mathbb{R}^{2d}}e^{\alpha | v|}f(t, z) \di z \le	\int_{\mathbb{R}^{2d}} e^{\alpha | v|}f_0(z) \di z, \vspace{6pt}\\
			\displaystyle \int_{\mathbb{R}^{2d}}e^{\alpha | v|^{\delta}}f(t, z) \di z \le	\int_{\mathbb{R}^{2d}} e^{\alpha | v|^{\delta}}f_0(z) \di z.
		\end{cases}
	\end{equation}
\item	
For a nonnegative locally bounded and increasing function $t \mapsto R_v(t)$, we have  
		\begin{equation} \label{B-5}
		\iint_{|v| \geq R_v(t)}   f(t, z) \di z \le \mathcal{M}_{e}(f_0, \alpha, \delta) e^{-\alpha \left(R_v(t)\right)^{\delta}}.
	\end{equation}
\end{enumerate}	
\end{lemma}
\begin{proof}
The proofs can be found in  Lemma 3.2 and Corollary 3.2 in reference \cite{HWKCS2026}.
\end{proof}
Next, we recall the propagation of spatial and velocity moments, and derive estimates on the mass of particles with sufficiently large velocities. 
\begin{lemma}\label{L2.4}
\emph{\cite{HWKCS2026}}
For $D\ge 2$ and $T \in (0, \infty]$, let $f \in {\mathcal C}([0, T);L^1(\mathbb{R}^{2d})) \cap L_+^\infty([0, T) \times \mathbb{R}^{2d})$ be a weak solution to \eqref{A-2} such that 
\[
{\mathcal M}_p(f_0, D, D) :=  \int_{\mathbb{R}^{2d}} (|x|^D+ |v|^D) f_0(z) \di z < \infty.
\]
Then, we have the following estimates: for $t \in [0, T)$, 
	\begin{align*}
	\begin{aligned} \label{C-1-101}
& (i)~\int_{\mathbb{R}^{2d}}v f(t, z) \di z =	\int_{\mathbb{R}^{2d}}vf_0(z) \di z. \\
& (ii)~ \int_{\mathbb{R}^{2d}} xf(t, z) \di z  =	\int_{\mathbb{R}^{2d}}x f_0(z) \di z  + t \int_{\mathbb{R}^{2d}}vf_0(z) \di z. \\
& (iii)~ \int_{\mathbb{R}^{2d}} | v|^{D}f(t, z) \di z \le	\int_{\mathbb{R}^{2d}} | v|^{D}f_0(z) \di z. \\
& (iv)~\int_{\mathbb{R}^{2d}} | x|^{D}f(t, z) \di z   \le	\left[ \left(\int_{\mathbb{R}^{2d}} | x|^{D}f_0(z) \di z \right)^{\frac{1}{D}}+ 
\left(\int_{\mathbb{R}^{2d}} | v|^{D}f_0(z) \di z \right)^{\frac{1}{D}}t\right ]^{D}.
	\end{aligned}
	\end{align*}
\end{lemma}
\begin{proof}
We refer to [Lemma 3.1, \cite{HWKCS2026}] for detailed proofs. 
\end{proof}

\section{Description of two main results} \label{sec:3}
\setcounter{equation}{0}
In this section, we briefly state our main results on the finite-time weak stability of the KCS model \eqref{A-2} and finite-time random mean-field limit as a direct application of weak stability. For a given weak solution $f \in  {\mathcal C}([0, T);L^1(\mathbb{R}^{2d})) \cap L_+^\infty([0, T) \times \mathbb{R}^{2d})$ to \eqref{A-2}, recall that $\di \mu_t (z) := f(t,z) \di z$ is the associated measure-valued solution to KCS model \eqref{A-2} with the density $f$.
\subsection{Finite-time weak stability} \label{sec:3.1}
In this subsection, we are ready to provide our first main result on weak stability in the following theorem. 
\begin{theorem}\label{T3.1}
For \(T>0\),~\(M>0\),~\(\alpha>0\),  assume that initial data $f_0$ and $g_0$ satisfy the boundedness moment condition:
	\begin{equation}\label{B-5-1}
	\max\left\{\int_{\mathbb R^{2d}}\bigl(|x|^2+e^{\alpha |v|}\bigr) f_0(z) \di z ,\quad \int_{\mathbb R^{2d}}\bigl(|x|^2+e^{\alpha |v|}\bigr) g_0(z) \di z \right\}\le M,
	\end{equation}
and let $\mu$ and $\nu$ be two measure-valued solutions to \eqref{A-2} on \([0,T]\) with densities $f, g \in {\mathcal C}([0, T);L^1(\mathbb{R}^{2d})) \cap L_+^\infty([0, T) \times \mathbb{R}^{2d})$, respectively. Then, there exists a continuous, concave, and increasing function $G_T$ such that 
\begin{equation} \label{B-5-1-1}
\sup_{0\le t\le T}W_1(\mu_t,\nu_t) \le G_T \Big (W_1(\mu_0, \nu_0) \Big).
\end{equation}
\end{theorem}
\begin{proof}
Since the proof is very lengthy, we provide an outline of the proof for weak stability \eqref{B-5-1-1} in four steps: 
\vspace{0.1cm}
\begin{itemize}
\item
Step A: For given points $z$ and $\bar z$ in the phase space $\bbr^{2d}$, we set $(X_f(t), V_f(t))$ and  $(X_g(t), V_g(t))$  to be forward particle trajectories  issued from $z$ and $\bar z$ via the mean-field force due to mass distributions $f \di z$ and $g \di z$, respectively, and $\pi_0(z, \bar z)$ is a coupling of the measures $f_0 \di z$ and $g_0 \di z$. In this setting, we introduce a functional $\Delta$ measuring the deviations of two particle trajectories:
\[ 
\Delta(t)
	:=
	\iint_{\mathbb R^{4d}}
	\Bigl(
	|X_f(t,z)-X_g(t,\bar z)|
	+
	|V_f(t,z)-V_g(t,\bar z)|
	\Bigr)
	\,\di \pi_0(z,\bar z). \]
Then, it is easy to see
\begin{equation} \label{B-5-1-2}
	W_1(\mu_t, \nu_t )\le \Delta(t),
	\quad \forall~t \in [0, T). 
\end{equation}	
See Section \ref{sec:4.1} for details. 
\vspace{0.2cm}
\item
Step B: We derive a generalized Gr\"onwall type differential inequality:~For any $R > 0$, there exist positive constants $C_{T, M}$ and $\lambda$ such that 
\begin{equation}  \label{New-1}
	\frac{\di}{\di t}\Delta(t)
	\le
	C_{T,M} \Big( (1+R)\Delta(t)+ e^{-\lambda R} \Big),
	\qquad
	\text{for a.e. }t\in[0,T).
\end{equation}
See Proposition \ref{NP4.1}.

\vspace{0.2cm}

\item
Step C: From the above differential inequality \eqref{New-1}, we derive an estimate for $\Delta$:
\begin{equation} \label{B-5-1-3}
\sup_{0\le t <T}\Delta(t) \leq 
\begin{cases}
C_{T,M}\Delta(0)^{e^{-C_{T,M}T}}, \quad & \Delta(0) \leq 1, \vspace{4pt}\\
 2 e^{C_{T,M} T} \Delta(0), \quad & \Delta(0)  > 1.
 \end{cases}
\end{equation}
\item
Step D: Finally, we combine \eqref{B-5-1-2} and \eqref{B-5-1-3} and take an infimum over all couplings $\pi_0$ to derive the desired finite-time weak stability:
\[  \sup_{0 \leq t \le T } W_1(\mu_t, \nu_t ) \leq G_T \Big (W_1(\mu_0, \nu_0) \Big ), \]
for some concave and increasing function $G_T$.
\end{itemize}
In Section \ref{sec:4}, we provide detailed arguments for the above steps. 
\end{proof}
\begin{remark} \label{R3.1} We provide several comments regarding on the result. 
\begin{enumerate}
\item
The explicit ansatz for $G_T(\zeta)$ can be given as follows :~for some $\Lambda_{T, M} \in (0, 1)$, 
\[ 
G_T(\zeta) :=  \max \Big\{ C_{T,M},~ 2 e^{C_{T,M} T} \Big \} \times\begin{cases}
\zeta^{\Lambda_{T,M}} \quad &\zeta \leq 1, \\
\zeta, \quad & \zeta > 1.
\end{cases}
\]
We refer to Section \ref{sec:4.2} for the detailed construction of $G_T$. Then, it is easy to see that $G_T$ is a continuous, concave, and increasing function such that 
\begin{align*}
\begin{aligned}
& G_T(0) = 0, \quad \lim_{\zeta \to 0} G_T(\zeta) = 0, \quad  \lim_{\zeta \to \infty} G_T(\zeta) = \infty, \\
& (G_T(\zeta_2) - G_T(\zeta_1)) (\zeta_2 - \zeta_1) \geq 0, \quad \forall~\zeta_1, \zeta_2 \geq 0.
\end{aligned}
\end{align*}
\item
When the spatial support is unbounded, in the reference \cite{HWY}, under the suitable assumption for compact velocity support, the authors showed that 
\begin{enumerate}
\item
the diameter of the velocity support may remain constant for all time. 
\vspace{0.1cm}
\item
there is no uniform constant $G$ independent of $T$ for which a uniform-time stability \eqref{A-5} holds for the KCS model \eqref{A-2} for general data. Therefore, in the spatially extended setting, we do not expect the existing finite-time weak stability estimate for the KCS model \eqref{A-2} to be extended to a uniform-time weak stability for general initial data (see Proposition 2.3 and Section 2.3 of reference \cite{HWY}).
\end{enumerate}
\item Our approach employs the characteristic flow technique and therefore does not require densities of the solutions. The density assumption is included only to align with the weak solution theory developed in reference \cite{kinetic3}. However, the finite-time weak stability established herein extends naturally to measure-valued solutions without any additional regularity as long as the solution exists and the push-forward is well-defined.
\end{enumerate}
\end{remark}
\subsection{Finite-time random mean-field limit} \label{sec:3.2}
In this subsection, we return to finite-time random mean-field limit as an application of finite-time weak stability described in Theorem \ref{T3.1}. Unlike the stochastic mean-field limit \cite{Bolley2011}, there is no explicit external noises in the governing dynamics \eqref{A-1} and \eqref{A-2}. Our randomness comes from the random sampling of initial configurations for the particle CS model. Due to the random initial datum, solution is itself a random process. Hence the empirical measures is also a random variable. Therefore, the convergence of random empirical measure to the measure whose density function is a weak solution to \eqref{A-2} should be understood in a probabilistic sense. For this, we take the expectation for the $W_1$-distance between the empirical measure and one-particle distribution function to show the convergence as $N \to \infty$. In this sense, we use the terminology `` {\it random mean-field limit} " instead of stochastic mean-field limit.

Let $\{ (x_i,v_i) \}_{i=1}^N$ be a global solution to CS model \eqref{A-1}. Then, the empirical measure generated by the solution $\{ (x_i,v_i) \}_{i=1}^N$ is given by
\[
\mu^N_t =\frac1N\sum_{i=1}^N\delta_{(x_i(t),v_i(t))}, \quad t > 0,
\]
and it is itself a measure-valued solution to \eqref{A-2} with initial datum $\mu^N_{0}=\frac1N\sum_{i=1}^N\delta_{z_{0,i}}$. Therefore, Theorem~\ref{T3.1} can be applied directly to compare the empirical measure \(\mu^N_t\) with the kinetic solution issued from a limiting initial law. Moreover, our second main result deals with the i.i.d. sampling consequence in the following theorem.
\begin{theorem}\label{T3.2}
Suppose that initial data $f_0\in (L^1 \cap L^{\infty}_+)(\mathbb{R}^{2d})$ and $Z_{0}= \{ (x_{i,0},v_{i,0}) \}$ satisfy the following conditions:
\begin{align}
\begin{aligned} \label{B-6}
& (i)~\int_{\mathbb R^{2d}}\bigl(|x|^q+|v|^q+e^{\alpha |v|}\bigr)f_0(z) \di z <\infty, \quad \mbox{for some \(\alpha>0\) and some \(q>2\),} \vspace{4pt}\\
& (ii)~z_{i,0}=(x_{i,0},v_{i,0})~\mbox{is  independent r.v. with the common law $f_0,~~ \forall~i \in [N]$,} \vspace{4pt}\\
& (iii)~{\mathbb E}[W_1(\mu^N_{0}, \mu_0)] \le C N^{-\eta}, \quad \mbox{for some $\eta > 0$}, \\
\end{aligned}
\end{align}	
and let  $\mu_t$ and \(\mu^N_t\)  be the weak solution to \eqref{A-2} with initial datum $\di \mu_0=f_0\di z$ and the empirical measure generated by  \eqref{A-1} with initial configuration $Z_0$. Then,  for every finite $T>0$, there exist positive constants $C_T$ and $\eta_T$ such that
	\begin{equation}\label{B-7}
		\mathbb E\left[\sup_{0\le t\le T}W_1(\mu^N_t, \mu_t) \right]\le \frac{C_T}{N^{\eta_T}}.
	\end{equation}
\end{theorem}
\begin{proof}
In what follows, we sketch the outline of the proof. Details can be found in Section \ref{sec:5}. Let $T$ be any positive number.  Suppose the conditions in \eqref{B-6} hold. 
\begin{itemize}
\item
Step A: For a given initial datum $f_0$ satisfying \eqref{B-6}, we set 
\[ 
Y(z):=|x|^2+e^{\frac{\alpha}{2}|v|}, \quad 
M_0:=\int_{\mathbb R^{2d}} Y(z)f_0(z) \di z <\infty, \quad M := 2M_0. 
\]
Then, we introduce a good event ${\mathcal G}_N$:
\[
{\mathcal G}_N:=
\left\{
\frac1N\sum_{i=1}^NY(z_{i,0}) \le M, \quad \forall~i \in [N]
\right\}.
\]
On the good event \(\mathcal G_N\), both \(\mu^N_{0} \) and \(f_0\) satisfy the condition \eqref{B-5-1} in Theorem~\ref{T3.1}, with \(\alpha_0 = \frac{\alpha}{2}\) in place of
\(\alpha\), and with common moment bound $M:=2M_0$. Then, it follows from Theorem \ref{T3.1} that 
\[
	\sup_{0\le t\le T}W_1(\mu^N_t, \mu_t)
	\le G_{T} \Big(W_1(\mu_{0}^N,~\mu_0) \Big) \quad \mbox{on}~{\mathcal G}_N.
\]
Using this stability estimate and decaying condition on initial data $\eqref{B-6}_3$, we can derive 
\begin{equation} \label{B-8}
\mathbb E\left[
	\mathbf 1_{\mathcal G_N}
	\sup_{0\le t\le T}W_1(\mu^N_t, \mu_t)
	\right]
	\le C_{T,M }N^{-\eta \Lambda_{T,M}}. 
\end{equation}
Moreover, by tedious technical estimate, the bad event ${\mathcal B}_N = {\mathcal G}_N^c$ satisfies
\begin{equation} \label{B-9}
\mathbb P(\mathcal B_N) \le C N^{-(p_0-1)}
\end{equation}
for some $p_0:=\min\left\{\frac q2,2\right\}>1$.  We refer to Lemma \ref{L5.1} for details. 
\vspace{0.2cm}

\item
Step B: We show that there exists a random variable ${\mathcal R}_N$ which depends on the random initial configuration $Z_0 := \{ (x_{i,0}, v_{i,0})\}$ such that 
\begin{equation} \label{B-10}
\sup_{0\le t\le T}W_1(\mu^N_t, \mu_t)
	\le
	\mathcal R_N \quad \quad \mbox{and} \quad 	\sup_{N\ge1}\mathbb E\bigl[\mathcal R_N^2\bigr]<\infty. 
\end{equation}
See Lemma \ref{L5.3} for details. 
\vspace{0.2cm}

\item
Step C: We use \eqref{B-8}, \eqref{B-9}, and \eqref{B-10} to derive the desired estimate:
\begin{align*}
\begin{aligned}  \label{B-11}
 \mathbb E\left[
	\sup_{0\le t\le T}W_1(\mu^N_t, \mu_t)
	\right] &=
	\mathbb E\left[
	\mathbf 1_{\mathcal G_N}
	\sup_{0\le t\le T}W_1(\mu^N_t, \mu_t)
	\right] +
	\mathbb E\left[
	\mathbf 1_{\mathcal B_N}
	\sup_{0\le t\le T}W_1(\mu^N_t, \mu_t)
	\right] \\
	&\lesssim C_{T,M} N^{-\eta \Lambda_{T,M}} + C_T N^{-\frac{p_0-1}{2}}.
\end{aligned}
\end{align*}
For all detailed arguments, we refer to Section \ref{sec:5.2}.
\end{itemize}
\end{proof}
\begin{remark}
The rate $\eta$ in $\eqref{B-6}_3$ is guaranteed by [Theorem 1, \cite{FG}].
\end{remark}
\vspace{0.2cm}

In the following two sections, we provide detailed proofs of Theorem \ref{T3.1} and Theorem \ref{T3.2}.
\section{Weak stability in $1$-Wasserstein metric}\label{sec:4}
\setcounter{equation}{0}
For the weak stability, we first adopt a procedure similar to that used in the uniqueness of weak solution in reference \cite{HWKCS2026}. Then, we use Osgood-type argument to derive the desired estimates.

\subsection{Preparatory estimates} \label{sec:4.1}
Suppose the initial data $f_0$ and $g_0$ satisfy positivity, boundedness and suitable decay conditions, respectively:
\begin{equation}  \label{C-1}
\begin{cases}
\displaystyle f_0,~g_0 \in (L^1 \cap L_+^\infty)(\mathbb{R}^{2d}), \quad (|x|^2 + e^{\alpha|v|})f_0,~(|x|^2 + e^{\alpha|v|}) g_0 \in L^1(\mathbb{R}^{2d}), \vspace{6pt}\\
\displaystyle \max\left\{\int_{\mathbb R^{2d}}\bigl(|x|^2+e^{\alpha |v|}\bigr) f_0(z)\di z,\quad \int_{\mathbb R^{2d}}\bigl(|x|^2+e^{\alpha |v|}\bigr) g_0(z) \di z \right\}\le M,
\end{cases}
\end{equation}
and let $f$ and $g$ be weak solutions to \eqref{A-2} whose existence and uniqueness are guaranteed by Proposition \ref{P2.2}. Then, we use  Lemma \ref{L2.3} to see 
				\begin{align*}
				\int_{\mathbb{R}^{2d}} e^{\alpha |v|}f(t, z) \di z \le \int_{\mathbb{R}^{2d}} e^{\alpha |v|}f_0(z) \di z, \quad \int_{\mathbb{R}^{2d}} e^{\alpha |v|} g(t, z) \di z \le \int_{\mathbb{R}^{2d}} e^{\alpha |v|} g_0(z) \di z.
			\end{align*}
Let $\pi_0 \in\Pi(f_0 \di z,g_0 \di z)$  be an arbitrary initial coupling of measures $f_0 \di z$ and $g_0 \di z$. For $z=(x,v),~~\bar z=(\bar x,\bar v)$, we set 
\[
Z_f(t,z):=(X_f(t,z),V_f(t,z)),
\qquad
Z_g(t,\bar z):=(X_g(t,\bar z),V_g(t,\bar z))
\]
to be the forward characteristic flows associated with \(f\) and \(g\), respectively:
\begin{equation}
\begin{cases} \label{C-2}
	\dot X_f(t,z)=V_f(t,z),\\[1mm]
	\dot V_f(t,z)=L[f]\bigl(t,X_f(t,z),V_f(t,z)\bigr), \\
	Z_f(0,z)=z,
\end{cases}
~~
\begin{cases}
	\dot X_g(t,\bar z)=V_g(t,\bar z),\\[1mm]
	\dot V_g(t,\bar z)=L[g]\bigl(t,X_g(t,\bar z),V_g(t,\bar z)\bigr),\\[1mm]
	Z_g(0,\bar z)=\bar z.
\end{cases}
\end{equation}
Then, we have
\[
f(t)=Z_f(t,\cdot)_{\#}f_0,
\quad
g(t)=Z_g(t,\cdot)_{\#}g_0.
\]
Now, we define
\begin{equation}\label{C-4}
	\Delta_{\pi}[f,g](t)
	:=
	\iint_{\mathbb R^{4d}}
	\Bigl(
	|X_f(t,z)-X_g(t,\bar z)|
	+
	|V_f(t,z)-V_g(t,\bar z)|
	\Bigr)
	\,\di \pi_0(z,\bar z).
\end{equation}
From now on,  we use handy notation:
\[
 \Delta(t):=\Delta_{\pi}[f,g](t), \quad t > 0.
\]
Then, it is easy to see
\[ 
 \Delta(0)
	=
	\iint_{\bbr^{4d}}
	\big(|x-\bar x| + |v-\bar v|\big) \di \pi_0(z,\bar z).
\]
Since
\[
\bigl(Z_f(t,\cdot),Z_g(t,\cdot)\bigr)_{\#}\pi_0
\in\Pi(f(t) \di z,g(t) \di {\bar z}), \quad \di \mu_t := f(t) \di z, \quad \di \nu_t := g(t) \di {\bar z},
\]
we have
\begin{equation}\label{C-5}
	W_1(\mu_t, \nu_t )\le \Delta(t),
	\qquad 0\le t\le T.
\end{equation}
\subsubsection{Deviation of two forces along particle trajectories} \label{sec:4.1.1} 
In this part, we estimate the global effect due to force differences in terms of $\Delta$:
\[
\iint_{\bbr^{2d}}
	\Bigl |
	L[f]\bigl(t,X_f(t,z),V_f(t,z)\bigr)
	-
	L[g]\bigl(t,X_g(t,\bar z),V_g(t,\bar z)\bigr)
	\Bigr|
	\,\di \pi_0(z,\bar z).
\]
Let \((z_*,\bar z_*)\) be an independent copy of \((z,\bar z)\) with the same law
\(\pi_0 \). For notational simplicity, we use the following handy notations:
\begin{align}
\begin{aligned} \label{C-5-0}
& X_f=X_f(t,z),\quad V_f=V_f(t,z),\quad
X_g=X_g(t,\bar z),\quad V_g=V_g(t,\bar z), \\
& X_{f,*}=X_f(t,z_*),\quad V_{f,*}=V_f(t,z_*),\quad
X_{g,*}=X_g(t,\bar z_*),\quad V_{g,*}=V_g(t,\bar z_*).
\end{aligned}
\end{align}
Since 
 \(f(t)=Z_f(t)_{\#}f_0\) \ \mbox{and} \ \(g(t)=Z_g(t)_{\#}g_0\), 
we may write
\begin{align}
\begin{aligned} \label{C-5-1}
	&L[f](X_f,V_f)-L[g](X_g,V_g)
	\\
	&\hspace{0.2cm} =
	\kappa
	\iint_{\bbr^{4d}}
	\Bigl[
	\phi(|X_f-X_{f,*}|)(V_{f,*}-V_f)
	-
	\phi(|X_g-X_{g,*}|)(V_{g,*}-V_g)
	\Bigr]
	\,\di \pi_0(z_*,\bar z_*).
\end{aligned}
\end{align}
Note that the integrand 
\begin{align}
\begin{aligned} \label{C-5-2}
	&\phi(|X_f-X_{f,*}|)(V_{f,*}-V_f)
	-
	\phi(|X_g-X_{g,*}|)(V_{g,*}-V_g)
	\\
	& \hspace{1cm} =
	\phi(|X_f-X_{f,*}|)
	\Bigl[
	(V_{f,*}-V_f)-(V_{g,*}-V_g)
	\Bigr] \\
	& \hspace{1.4cm} +
	\Bigl[
	\phi(|X_f-X_{f,*}|)
	-
	\phi(|X_g-X_{g,*}|)
	\Bigr]
	(V_{g,*}-V_g)
	\\
	& \hspace{1cm} =: {\mathcal I}_{11} +  {\mathcal I}_{12}.
\end{aligned}
\end{align}
\begin{lemma} \label{L4.1} 
The terms ${\mathcal I}_{1i},~i=1,2$ satisfy the following estimates:~for $R \geq 1$ and for some $\alpha_1 \in (0, \alpha)$, 
\begin{align*}
\begin{aligned}
& (i)~\iiiint_{\bbr^{8d}} | {\mathcal I}_{11}|
	\di \pi_0(z_*,\bar z_*)\,\di \pi_0(z,\bar z)
	\le 2 \|\phi\|_{L^\infty}  \Delta.    \\
& (ii)~\iiiint_{\bbr^{8d}} | {\mathcal I}_{12}|
	\di \pi_0(z_*,\bar z_*)\,\di \pi_0(z,\bar z) \leq 2\|\phi'\|_{L^\infty} \Big( \sqrt{M} \Delta +   R\Delta + \frac{4M}{\alpha_1} \sqrt{ 1 + \frac{2T^2}{\alpha^2}} e^{-\frac{\alpha - \alpha_1}{2} R} \Big).
\end{aligned}
\end{align*}
\end{lemma}
\begin{proof} 
\noindent (i)~In \eqref{C-5-2}, we use \(0\le\phi\le \|\phi\|_{L^\infty}\) to see
\[
| {\mathcal I}_{11}|
\le
\|\phi\|_{L^\infty}
\Bigl(
|V_f-V_g|+|V_{f,*}-V_{g,*}|
\Bigr).
\]
This and \eqref{C-4} imply
\begin{equation}\label{C-6}
	\iiiint_{\bbr^{8d}} | {\mathcal I}_{11}|\,
	\di \pi_0(z_*,\bar z_*)\,\di \pi_0(z,\bar z)
	\le 2 \|\phi\|_{L^\infty} \Delta(t).
\end{equation}

\noindent (ii) Since the proof for the second assertion is very lengthy, we leave it in Appendix \ref{App-A}.
\end{proof}
Next, we present the effect of the force differences along particle trajectories.
\begin{proposition}  \label{P4.1}
For any \(R>0\) and some \(\alpha_1\in(0,\alpha)\), the following estimate holds.
\begin{align*}
\begin{aligned}
& \iint_{\bbr^{4d}}
	\Bigl |
	L[f]\bigl(t,X_f(t,z),V_f(t,z)\bigr)
	-
	L[g]\bigl(t,X_g(t,\bar z),V_g(t,\bar z)\bigr)
	\Bigr|
	\,\di \pi_0(z,\bar z) \\
& \hspace{1.5cm} \leq 2 \kappa \|\phi \|_{W^{1,\infty}} \Big[ (1 + \sqrt{M}) \Delta +   R\Delta +   \frac{4M}{\alpha_1} \sqrt{ 1 + \frac{2T^2}{\alpha^2}} e^{-\frac{\alpha - \alpha_1}{2} R} \Big ].
\end{aligned}
\end{align*}
\end{proposition}
\begin{proof} We use  \eqref{C-5-1} and Lemma \ref{L4.1} to see
\begin{align*}
\begin{aligned} 
&  \iint_{\bbr^{4d}} \Big| L[f](X_f,V_f)-L[g](X_g,V_g) \Big| \,\di \pi_0(z,\bar z)  \\
&\hspace{0.5cm} \leq
	\kappa
	\iiiint_{\bbr^{8d}}
	\Big |
	\phi(|X_f-X_{f,*}|)(V_{f,*}-V_f)
	-
	\phi(|X_g-X_{g,*}|)(V_{g,*}-V_g)
	\Big |
	\,\di \pi_0(z_*,\bar z_*) \,\di \pi_0(z,\bar z)  \\
& \hspace{0.5cm} \leq 2 \kappa \|\phi\|_{L^\infty}  \Delta +  2 \kappa \|\phi'\|_{L^\infty} \Big( \sqrt{M} \Delta +  R\Delta +  \frac{4M}{\alpha_1} \sqrt{ 1 + \frac{2T^2}{\alpha^2}}e^{-\frac{\alpha - \alpha_1}{2} R} \Big) \\
& \hspace{0.5cm} \leq 2 \kappa \|\phi \|_{W^{1,\infty}} \Big[ (1 + \sqrt{M}) \Delta +   R\Delta +  \frac{4M}{\alpha_1} \sqrt{ 1 + \frac{2T^2}{\alpha^2}} e^{-\frac{\alpha - \alpha_1}{2} R} \Big ].
\end{aligned}
\end{align*}	
\end{proof}	
\subsubsection{Derivation of differential inequality for $\Delta$} \label{sec:4.1.2}
In this part, we derive the differential inequality for $\Delta = \Delta(t)$. Since the map
\(r\mapsto |r|\) is not differentiable at $r = 0$, the following computation
can be justified by replacing \(|\xi|\) with \((|\xi|^2+\varepsilon^2)^{1/2}\),
deriving estimates uniformly in \(\varepsilon>0\), and then letting
\(\varepsilon\downarrow0\). 
\begin{proposition}  \label{NP4.1}
For every \(R > 0\) and $T \in (0, \infty)$, the functional $\Delta$ satisfies 
\begin{equation}\label{C-9}
	\frac{\di  \Delta(t)}{\di t}
	\le
	C_{T,M} \Big( (1+R)\Delta(t)
	+ e^{-\lambda R} \Big),
	\qquad
	\text{for a.e. }t\in[0,T].
\end{equation}
Here, $C_{T, M}$ and $\lambda$ are positive constants defined as follows:
\begin{equation} \label{C-9-0}
C_{T, M} :=  \max \Bigg \{  1 +  2 \kappa \|\phi \|_{W^{1,\infty}}  (1 + \sqrt{M}),~\frac{8 \kappa \|\phi \|_{W^{1,\infty}} M}{\alpha_1} \sqrt{ 1 + \frac{2T^2}{\alpha^2}} \Bigg \}, \quad \lambda :=  \frac{\alpha - \alpha_1}{2} > 0.   
\end{equation}
\end{proposition}
\begin{proof}
\noindent For a.e. \(t\in[0,T]\), we have
\begin{align} 
\begin{aligned} \label{C-9-1}
 \frac{\di}{\di t}\Delta(t)
	&\le
	\iint_{\bbr^{4d}} \Big |V_f(t,z)-V_g(t,\bar z) \Big|\,\di \pi_0(z,\bar z)
	\\
	&+
	\iint_{\bbr^{4d}}
	\Bigl |
	L[f]\bigl(t,X_f(t,z),V_f(t,z)\bigr)
	-
	L[g]\bigl(t,X_g(t,\bar z),V_g(t,\bar z)\bigr)
	\Bigr|
	\,\di \pi_0(z,\bar z)
	\\
	&=: {\mathcal I}_{21}(t)+ {\mathcal I}_{22}(t).
\end{aligned}
\end{align}
Then, we use the definition of $\Delta$ and Proposition \ref{P4.1} to see
\begin{align}
\begin{aligned} \label{C-9-2}
& {\mathcal I}_{21}(t) \leq \Delta(t), \\
&  {\mathcal I}_{22}(t) \leq 2 \kappa \|\phi \|_{W^{1,\infty}} \Bigg[ (1 + \sqrt{M}) \Delta(t) +   R\Delta(t) +  \frac{4M}{\alpha_1} \sqrt{ 1 + \frac{2T^2}{\alpha^2}} e^{-\frac{\alpha - \alpha_1}{2} R} \Bigg].
\end{aligned}
\end{align}
Finally, we use \eqref{C-9-1} and \eqref{C-9-2} to get the desired estimate:
\begin{align*}
\begin{aligned}
 \frac{\di}{\di t}\Delta(t) &\leq \Delta(t) + 2 \kappa \|\phi \|_{W^{1,\infty}} \Bigg[ (1 + \sqrt{M}) \Delta(t) +   R\Delta(t) +  \frac{4M}{\alpha_1} \sqrt{ 1 + \frac{2T^2}{\alpha^2}} e^{-\frac{\alpha - \alpha_1}{2} R} \Bigg ] \\
 &= \Big( 1 +  2 \kappa \|\phi \|_{W^{1,\infty}}  (1 + \sqrt{M})            \Big) \Delta(t) +  2 \kappa \|\phi \|_{W^{1,\infty}} R \Delta(t)  \\
 &+  \frac{8 \kappa \|\phi \|_{W^{1,\infty}} M}{\alpha_1} \sqrt{ 1 + \frac{2T^2}{\alpha^2}} e^{-\frac{\alpha - \alpha_1}{2} R} \\
 & \leq  \Big( 1 +  2 \kappa \|\phi \|_{W^{1,\infty}}  (1 + \sqrt{M})            \Big) (1 + R) \Delta(t) +  \frac{8 \kappa \|\phi \|_{W^{1,\infty}} M}{\alpha_1} \sqrt{ 1 + \frac{2T^2}{\alpha^2}}  e^{-\frac{\alpha - \alpha_1}{2} R} \\
 & \leq \underbrace{\max \Big \{  \Big( 1 +  2 \kappa \|\phi \|_{W^{1,\infty}}  (1 + \sqrt{M}) \Big),~  \frac{8 \kappa \|\phi \|_{W^{1,\infty}} M}{\alpha_1} \sqrt{ 1 + \frac{2T^2}{\alpha^2}}        \Big \}}_{=: C_{T, M}} \\
 & \times \Big(  (1 + R) \Delta(t) +   e^{-\frac{\alpha - \alpha_1}{2} R} \Big).
\end{aligned}
\end{align*}
\end{proof}

\subsection{Proof of Theorem \ref{T3.1}} \label{sec:4.2}
In this subsection, we finish the proof of Theorem \ref{T3.1} for the last two steps. \newline

\noindent $\bullet$ \textbf{Step C (Conversion of \eqref{C-9} into an Osgood differential inequality)}:~We convert \eqref{C-9} into an Osgood type differential inequality. Since $R > 0$ is arbitrary constant, we  first replace $\lambda R$ by $R$ and change the constant $C_{T, M}$ accordingly to rewrite \eqref{C-9} as
\begin{equation}\label{C-9-3}
	\frac{\di}{\di t}\Delta(t)
	\le
	C_{T, M} \Big( (1+R)\Delta(t)
	+ e^{-R} \Big).
\end{equation}
In what follows, we denote the constant multiples of $C_{T, M}$ by the same constant $C_{T, M}$ as long as there is no confusion. Note that the differential inequality \eqref{C-9-3} holds for $R = 1$:
\begin{equation}\label{C-9-4}
	\frac{\di}{\di t}\Delta(t)
	\le
	 C_{T, M} \Big( \Delta(t)
	+ e^{-1} \Big).
\end{equation}
Then, we apply Gr\"onwall's lemma to \eqref{C-9-4} to get a bound for $\Delta$ in the time-interval $[0, T)$:
\begin{equation}\label{C-9-5}
	\sup_{0\le t\le T}\Delta(t)
	\le
	e^{C_{T,M} T} \Big( \Delta(0) + \frac{1}{e} \Big) - \frac{1}{e} \le e^{C_{T,M} T} \Big( \Delta(0) + \frac{1}{e} \Big).
\end{equation}
Now, we claim that 
\begin{equation} \label{C-9-5-11}
\sup_{0\le t\le T}\Delta(t) \leq 
\begin{cases}
C_{T,M}\Delta(0)^{e^{-C_{T,M}T}}, \quad & \Delta(0) \leq 1, \vspace{4pt}\\
 2 e^{C_{T,M} T} \Delta(0), \quad & \Delta(0)  > 1.
 \end{cases}
\end{equation}
{\it Proof of  \eqref{C-9-5-11}}: We consider the following two cases:
\[ \mbox{Either}~~ \Delta(0) \leq 1 \quad \mbox{or} \quad  \Delta(0) > 1. \]
\noindent $\diamond$~{\bf Case 1}: Suppose that 
\[ \Delta(0) \leq 1. \]
Next, we use a rough estimate \eqref{C-9-5} to take a sufficiently large positive constant \(K_{T,M}>0\) so that 
\begin{equation} \label{C-9-6}
a(t):=\frac{\Delta(t)}{K_{T,M}}\le e^{-1}, \quad \forall~t \in [0, T).
\end{equation}
We divide \eqref{C-9-3} by \(K_{T,M}\) to obtain
\begin{align}
\begin{aligned} \label{C-9-7}
	{\dot a}(t) &= \frac{{\dot \Delta}(t)}{K_{T,M}} \leq \frac{C_{T, M}}{K_{T, M}}  \Big(  (1+R)\Delta(t)
	+ e^{-R}\Big) \\
	&=
	C_{T,M} \Big( (1+R)a(t)
	+ K_{T,M}^{-1}e^{-R} \Big),
\end{aligned}
\end{align}
where we used \eqref{C-9-6}. Since \eqref{C-9-7} holds for every \(R > 0\), we can choose a time-dependent $R = R(t)$:
\begin{equation} \label{C-9-8}
R(t):=-\log a(t)\ge1, \quad \mbox{i.e.,} \quad e^{-R(t)}=a(t).
\end{equation}
\noindent If necessary, by enlarging \(K_{T,M}\), we  use \eqref{C-9-7} and \eqref{C-9-8} to obtain
\[
	a'(t)
	\le
	C_{T,M}a(t)(1-\log a(t)).
\]
Since \(0<a(t)\le e^{-1}\), we have
\[
1-\log a(t)\le -2\log a(t).
\]
Thus, we redefine $C_{T,M}$ as $2C_{T,M}$ as long as there is no confusion to get
\[
	a'(t)\le -C_{T,M}a(t)\log a(t),
\]
or equivalently,
\[
\frac{\di}{\di t}\log\bigl(-\log a(t)\bigr)\ge -C_{T,M}.
\]
We integrate the above relation over \([0,t]\) to get 
\[
-\log a(t)
\ge
e^{-C_{T,M}t}\bigl(-\log a(0)\bigr).
\]
Hence, we have
\begin{equation}\label{C-9-9}
	a(t)\le a(0)^{e^{-C_{T,M}t}},
	\qquad 0\le t\le T.
\end{equation}
It follow from \eqref{C-9-6} and \eqref{C-9-9} that 
\begin{align}
\begin{aligned} \label{C-9-10}
\Delta(t) &= K_{T,M} a(t) \leq  K_{T,M}  a(0)^{e^{-C_{T,M}t}}  =  K_{T,M}  \Big( \frac{\Delta(0)}{K_{T,M}} \Big)^{e^{-C_{T,M}t}} \\
	&=
	K_{T,M}^{1-e^{-C_{T,M}t}}
	\Delta(0)^{e^{-C_{T,M}t}}
	\le
	C_{T,M}\Delta(0)^{e^{-C_{T,M}T}},
	\qquad 0\le t\le T.
\end{aligned}
\end{align}
Note that the right-hand side of \eqref{C-9-10} cannot be controlled by a linear bound $\Delta(0)$, since 
\[
\Delta(0) \le \Delta(0)^{e^{-C_{T,M}T}}.
\]

\vspace{0.2cm}

\noindent $\diamond$~{\bf Case 2}: Suppose that 
\begin{equation} \label{C-29-1}
 \Delta(0) > 1.
\end{equation}
By \eqref{C-9-5} and \eqref{C-29-1}, we have
\begin{align}
\begin{aligned} \label{C-29-2}
	\sup_{0\le t\le T}\Delta(t) &\le e^{C_{T,M} T} \Big( \Delta(0) + \frac{1}{e} \Big) \leq 2 e^{C_{T,M} T} \Delta(0).
\end{aligned}
\end{align}

\vspace{0.2cm}

\noindent $\bullet$ \textbf{Step D (Derivation of weak stability)}:~Note that 
\[ \Delta(0) = \iint_{\bbr^{4d}} |z-\bar z|\,\di \pi_0(z,\bar z), \quad  \Lambda_{T,M}:=e^{-C_{T,M}T} \in (0, 1).\]
Then, the relation \eqref{C-9-5-11} can be rewritten as 
\begin{align}
\begin{aligned} \label{C-9-5-1}
& \sup_{0\le t\le T} \Delta_{\pi}[f,g](t)  \\
& \hspace{1cm} \leq  \max \Big\{ C_{T,M},~ 2 e^{C_{T,M} T} \Big \}\times
\begin{cases}
\displaystyle  \Big( \iint_{\bbr^{2d}} |z-\bar z|\,\di \pi_0(z,\bar z) \Big)^{\Lambda_{T, M}}, \quad & \Delta(0) \leq 1, \vspace{6pt}\\
\displaystyle  \iint_{\bbr^{2d}} |z-\bar z|\,\di \pi_0(z,\bar z), \quad & \Delta(0)  > 1.
 \end{cases}
\end{aligned}
\end{align}
Now, we combine \eqref{C-5} and \eqref{C-9-5-1} to see that for $t \in [0, T]$, 

\begin{align}
\begin{aligned} \label{C-10}
W_1(\mu_t, \nu_t ) &\le \Delta_{\pi}[f,g](t) \\
&\leq \max \Big\{ C_{T,M},~ 2 e^{C_{T,M} T} \Big \}\times
\begin{cases}
\displaystyle  \Big( \iint_{\bbr^{2d}} |z-\bar z|\,\di \pi_0(z,\bar z) \Big)^{\Lambda_{T, M}}, \quad & \Delta(0) \leq 1, \vspace{6pt}\\
\displaystyle   \iint_{\bbr^{2d}} |z-\bar z|\,\di \pi_0(z,\bar z), \quad & \Delta(0)  > 1.
 \end{cases}
\end{aligned}
\end{align}
Finally, we take the infimum over all \(\pi_0 \in\Pi(f_0,g_0)\) in the right-hand side of \eqref{C-10} to find 
\begin{align}
\begin{aligned} \label{C-11}
W_1(\mu_t, \nu_t ) &\le \Delta_{\pi}[f,g](t) \\
&\leq \max \Big\{ C_{T,M},~ 2 e^{C_{T,M} T} \Big \}\times
\begin{cases}
\displaystyle W_1(\mu_0, \nu_0)^{\Lambda_{T,M}}, \quad & W_1(\mu_0, \nu_0) \leq 1, \vspace{6pt}\\
\displaystyle W_1(\mu_0, \nu_0), \quad & W_1(\mu_0, \nu_0)  > 1.
 \end{cases}
\end{aligned}
\end{align}
Now, we define a continuous function $G_T$: for some $\Lambda_{T, M} \in (0, 1)$, 
\begin{equation} \label{C-11-1}
G_T(\zeta) :=  \max \Big\{ C_{T,M},~ 2 e^{C_{T,M} T} \Big \}\times \begin{cases}
\zeta^{\Lambda_{T,M}} \quad &\zeta \leq 1, \vspace{6pt}\\
\zeta, \quad & \zeta > 1.
\end{cases}
\end{equation}
Then, the estimate \eqref{C-11} yields a weak-stability in $W_1$-metric:
\begin{equation} \label{C-12}
W_1(\mu_t, \nu_t ) \leq G_T \Big (W_1(\mu_0, \nu_0) \Big ).
\end{equation}
This completes the proof of Theorem \ref{T3.1}. \newline
\begin{remark}\label{R4.1}
Below, we briefly wrap up several crucial ingredients employed in the proof.  The main difficulty in the above stability argument comes from the absence of
	compact velocity support. In the compactly supported case, all velocities remain
	uniformly bounded on finite time intervals, and the alignment force is effectively
	Lipschitz along the characteristics. One can then close a standard Dobrushin-type
	estimate. However, in the current phase-spatially extended setting, this argument breaks down
	because the force difference produces the term
	\begin{equation} \label{C-13}
	|V_g|\,|X_f-X_g|,
	\end{equation}
	where the factor \(|V_g|\) cannot be bounded uniformly. The key idea of the proof is to replace the missing uniform velocity bound by a
	velocity-truncation argument (see Case B.2 of Lemma \ref{L4.1} in Appendix \ref{App-A}). On the region \(\{|V_g|\le R\}\), the term \eqref{C-13} is
	controlled by \(R\Delta(t)\), while on the tail domain \(\{|V_g|>R\}\), the
	propagated exponential velocity moment gives an exponentially small error. This
	leads to the differential inequality with a control parameter $R > 0$:
	\[
	\frac{\di}{\di t}\Delta(t)
	\le
	C_{T,M} \Big( (1+R)\Delta(t) + e^{-\lambda R} \Big).
	\]
By choosing \(R\sim |\log\Delta(t)|\), one can derive an Osgood-type stability estimate or weak-stability estimate \eqref{C-12}, which is certainly weaker than the Lipschitz stability available under compact
	support, but it is sufficient to obtain finite-time continuous dependence in
	\(W_1\). This estimate is the key ingredient for the finite-time mean-field limit. Indeed,
	if the initial empirical measures \(\mu^N_{0} \) converge to \(f_0 \di x \di v\) in \(W_1\) and
	satisfy the same noncompact moment bound, then the empirical measures generated
	by the particle system satisfy
	\[
	\sup_{0\le t\le T}W_1(\mu^N_t, \mu_t)
	\le G_T \Big (W_1(\mu^N_0, \mu_0) \Big),
	\]
	where $\mu_t = f(t) \di z$ and $\mu_0 = f_0 \di z$.  Consequently, the convergence of the initial empirical measures implies the convergence of the empirical measure to the kinetic solution in any finite-time window. This also provides a new approach to the existence of solution to KCS model \eqref{A-2} in a phase-spatially extended setting.
\end{remark}

%
%
%
\section{Random mean-field limit}\label{sec:5}
\setcounter{equation}{0}
In this section, we study the rigorous derivation of `the finite-time random mean-field limit as the application of the weak stability in Theorem \ref{T3.1}.  \newline

For $T \in (0, \infty)$, let $\{ (x_i(t),v_i(t)) \}_{i=1}^N$ and $f = f(t, x,v)$ be global random processes to \eqref{A-1} with random initial data  $\{ (x_{i,0},v_{i,0}) \}_{i=1}^N$ and a global weak solution to \eqref{A-2} with deterministic datum $f_0$, respectively. We assume that the law of random process $z_{i,0} = (x_{i,0}, v_{i,0})$ is given by the same law $f_0$ for each $i \in [N]$.  Then, for such random process $\{ (x_i(t),v_i(t)) \}_{i=1}^N$,  we introduce the random empirical measure generated  from the random process $z_{i}(t) = (x_i(t), v_i(t))$:
\[
\mu^N_t =\frac1N\sum_{i=1}^N\delta_{z_i(t)}, \quad t \geq 0.
\]
Note that it is a path-wise measure-valued solution to \eqref{A-2} with random initial datum $\mu^N_{0}=\frac1N\sum_{i=1}^N\delta_{z_{i,0}}$.  In order to apply Theorem \ref{T3.1}, we need to consider a ``good event" which the assumption \eqref{B-5-1} holds. More precisely, for a fixed \(T>0\), we choose
\[
0<\alpha_0<\alpha, \quad \mbox{say} \quad \alpha_0= \frac{\alpha}{2},
\]
and we set 
\[
Y(z):=|x|^2+e^{\alpha_0|v|}, \quad 
M_0:=\int_{\mathbb R^{2d}}Y(z)  f_0(z) \di z <\infty, \quad M := 2M_0. 
\]
Note that 
\begin{equation} \label{D-0}
 \int_{\bbr^{2d}} Y(z)  \di \mu^N_0 = \frac{1}{N} \sum_{i=1}^{N} Y(z_{i,0}). 
 \end{equation}
For each \(N\), we define a good event consisting of sample points satisfying the condition in Theorem \ref{T3.1}:
\begin{equation} \label{D-1}
\mathcal G_N
:=
\left\{
\frac1N\sum_{i=1}^NY(z_{i,0}) \le M
\right\}.
\end{equation}
Therefore, on the good event ${\mathcal G}_N$, Theorem~\ref{T3.1} can be applied directly to compare the empirical measure \(\mu^N_t\) with the kinetic solution issued from a limiting initial law $f_0$ along the sample path. 
\subsection{Preparatory estimates} \label{sec:5.1}
In this subsection, we present probabilistic estimates for the ``good event ${\mathcal G}_N$" and its complement (bad event) ${\mathcal B}_N := {\mathcal G}_N^c$. These estimates will be crucially used in the proof of Theorem \ref{T3.2}. We set 
\[ \di \mu_t(x,v) := f(t,x,v) \di x \di v, \quad t > 0, \quad  \di \mu_0(x,v) := f_0(x,v) \di x \di v.\]
\begin{lemma} \label{L5.1}
Suppose that initial data $f_0\in (L^1 \cap L^{\infty}_+)(\mathbb{R}^{2d})$ and $Z_{0}= \{ z_{i,0} := (x_{i,0},v_{i,0}) \}$ satisfy the conditions \eqref{B-6}, and let  $f$ and \(\mu^N_t\)  be the weak solution to \eqref{A-2} and the empirical measure generated by  \eqref{A-1} with initial configuration $Z_0$, respectively. Then the good and bad events satisfy the following estimates:
\[  \mathbb E\left[
	\mathbf 1_{\mathcal G_N}
	\sup_{0\le t\le T}W_1(\mu^N_t, \mu_t)
	\right]
	\le C_{T,M }N^{-\eta \Lambda_{T,M}}, \quad \mathbb P(\mathcal B_N) \le C N^{-(p_0-1)}.
\]
Here, $p_0:=\min\left\{\frac q2,2\right\}.$
\end{lemma}
\begin{proof}
\noindent (i)~ Recall that on the good event \(\mathcal G_N\), we have
\begin{align*}
\begin{aligned}
& \int_{\mathbb R^{2d}}\bigl(|x|^2+e^{\alpha_0|v|}\bigr)\,\di \mu^{N}_0 (x,v)
=
\frac1N\sum_{i=1}^NY(Z_i^0)
\le 2M_0, \\
& \int_{\mathbb R^{2d}}\bigl(|x|^2+e^{\alpha_0|v|}\bigr) d\mu_0(x,v)
=
M_0
\le 2M_0.
\end{aligned}
\end{align*}
Hence, on \(\mathcal G_N\), both \(\mu^N_{0} \) and \(f_0\) satisfy the fully
noncompact moment assumption \eqref{B-5-1} in Theorem~\ref{T3.1}, with \(\alpha_0\) in place of
\(\alpha\), and with common moment bound $M:=2M_0.$ Then, it follows from Theorem \ref{T3.1} that 
\begin{equation}\label{D-3}
	\sup_{0\le t\le T}W_1(\mu^N_t, \mu_t)
	\le G_{T}\Big(W_1(\mu_{0}^N,~\mu_0)\Big ) \quad \mbox{on}\quad{\mathcal G}_N.
\end{equation}
Now, we multiply by \(\mathbf 1_{\mathcal G_N}\) to \eqref{D-3}, take an expectation and use $\mathbf 1_{\mathcal G_N}\le1$ to get
\[
\mathbb E\left[
\mathbf 1_{\mathcal G_N}
\sup_{0\le t\le T}W_1(\mu^N_t, \mu_t)
\right]
\le
\mathbb E\Big[
\mathbf 1_{\mathcal G_N} G_T \Big(W_1(\mu_{0}^{N}, \mu_0) \Big ) \Big] \leq \mathbb E \Big[ G_T \Big (W_1(\mu_{0}^{N}, \mu_0) \Big) \Big].
\]
Since the map $\zeta \mapsto G_T(\zeta)$ is concave, we use Jensen's inequality to get 
\begin{equation} \label{D-3-1}
\mathbb E\left[
\mathbf 1_{\mathcal G_N}
\sup_{0\le t\le T}W_1(\mu^N_t, \mu_t)
\right]
\le
\mathbb E\Big[ G_T \Big (W_1(\mu_{0}^{N}, \mu_0) \Big)
\Big] \leq G_T \Big( {\mathbb E} \Big [W_1(\mu_{0}^{N}, \mu_0) \Big] \Big ).
\end{equation}
Note that 
\[ N^{-\eta} < 1. \]
For the estimate of the right-hand side of \eqref{D-3-1}, we use  \eqref{C-11-1} and $\eqref{B-6}_3$ with $N^{-\eta} < 1$ to find 
\begin{align}
\begin{aligned} \label{D-4}
G_T \Big(\mathbb E \Big[ 
W_1(\mu^N_0, \mu_0) \Big ] \Big)  \leq  C \max \Big\{ C_{T,M},~ 2 e^{C_{T,M} T} \Big \} 
 N^{-\eta \Lambda_{T,M}}.
 \end{aligned}
\end{align}
Therefore, we combine \eqref{D-3-1} and \eqref{D-4} to obtain the desired estimate:
\[
\mathbb E\left[
\mathbf 1_{\mathcal G_N}
\sup_{0\le t\le T}W_1(\mu^N_t, \mu_t)
\right]  \leq C \max \Big\{ C_{T,M},~ 2 e^{C_{T,M} T} \Big \} 
 N^{-\eta \Lambda_{T,M}},
\]
where the constant \(C\) is independent of \(N\).  \newline

\noindent (ii)~Next, we show that the contribution of the bad event \(\mathcal B_N\) becomes negligible as $N \to \infty$:
\begin{equation} \label{D-4-1}
\mathbb P({\mathcal B}_N) = {\mathbb P} \left ( \left\{
\frac1N\sum_{i=1}^NY(z_{i,0}) > 2M_0
\right\} \right ).
\end{equation}
Since the $z_{i,0}$ are i.i.d, we only focus on $z_{1,0}$ for probabilistic estimates.
Now, we use \(q>2\) and \(\alpha_0<\alpha\) to see that there exists \(p_0\in(1,2]\) such that
\[
\mathbb E\bigl[Y(z_{1,0})^{p_0}\bigr]<\infty.
\]
One may take 
\[
\alpha_0=\alpha/2, \quad p_0:=\min\left\{\frac q2,2\right\}.
\]
Then, we have
\[ p_0> 1, \quad  2p_0\le q  \quad \mbox{and} \quad  p_0\alpha_0\le\alpha. \] 
Indeed, since \(p_0>1\), we use
\[
(a+b)^{p_0}\le 2^{p_0-1}(a^{p_0}+b^{p_0}),\qquad a,b\ge0,
\]
to see
\[
Y(z)^{p_0}
=
\left(|x|^2+e^{\alpha_0|v|}\right)^{p_0}
\le  2^{p_0-1} \left(|x|^{2p_0}+e^{p_0\alpha_0|v|}\right).
\]
 Hence, we have
\[
|x|^{2p_0}\le 1+|x|^q
\quad \mbox{and} \quad 
e^{p_0\alpha_0|v|}\le e^{\alpha|v|}.
\]
Consequently,
\[
Y(z)^{p_0}
\le 2^{p_0-1} \left(1+|x|^q+e^{\alpha|v|}\right).
\]
Therefore, one has 
\[
\mathbb E\bigl[Y(z_{1,0})^{p_0}\bigr]
=
\int_{\mathbb R^{2d}}Y(z)^{p_0} f_0(z) \di z
<\infty.
\]
We set
\[
\widetilde Y_i:=Y(z_{i,0})-M_0, \quad i \in [N].
\]
Then \((\widetilde Y_i)_{i=1}^N\) are independent centered random variables and $\mathbb E[|\widetilde Y_1|^{p_0}]<\infty.$ Since $\{ z_{i,0} \}$ are i.i.d., the shifted random variable $\{ \widetilde Y_i:=Y(z_{i,0})-M_0 \}$ is also i.i.d.. Moreover,
\[
\mathbb E[\widetilde Y_i]
=
\mathbb E[Y(z_{i,0})] -M_0
=
0.
\]
Since \(\mathbb E [Y(z_{1,0})^{p_0}]<\infty\), we also have
\begin{equation} \label{D-4-2}
\mathbb E[|\widetilde Y_1|^{p_0}] <\infty.
\end{equation}
Indeed, we have
\[
|\widetilde Y_1|^{p_0}
= \Big |Y(Z_1^0)-M_0 \Big |^{p_0}
\le 2^{p_0 - 1} \Big(Y(z_{1,0})^{p_0}+M_0^{p_0} \Big).
\]
We now use the von Bahr--Esseen inequality \cite{FP2017}. Since
\(1<p_0\le2\) and \(\widetilde Y_i\) are independent centered random variables,
\[
\mathbb E \left[ \left|\sum_{i=1}^N \widetilde Y_i\right|^{p_0} \right]
\le
C_{p_0}\sum_{i=1}^N\mathbb E \left[ |\widetilde Y_i|^{p_0} \right].
\]
Because the random variables are i.i.d., we have
\[
\sum_{i=1}^N\mathbb E[|\widetilde Y_i|^{p_0}]
=
N\mathbb E[|\widetilde Y_1|^{p_0}].
\]
Therefore, we have
\[
\mathbb E \left[ \left|
\frac1N\sum_{i=1}^N\widetilde Y_i
\right|^{p_0} \right]
=
\frac1{N^{p_0}}
\mathbb E \left[ \left|
\sum_{i=1}^N\widetilde Y_i
\right|^{p_0} \right]
\le
\frac{C_{p_0}}{N^{p_0}}
N\mathbb E[|\widetilde Y_1|^{p_0} ] = C_{p_0} N^{1-p_0} \mathbb E[|\widetilde Y_1|^{p_0} ].
\]
We use \eqref{D-4-2} and the above relation to find 
\begin{equation} \label{D-4-3}
\mathbb E \left[ \left|
\frac1N\sum_{i=1}^N\widetilde Y_i
\right|^{p_0} \right]
\le
C_{p_0} N^{-(p_0-1)}.
\end{equation}
Now, we use a generalized Markov's inequality and \eqref{D-4-3} to obtain
\begin{align}
\begin{aligned} \label{D-4-4}
\mathbb P(\mathcal B_N)
	&=
	\mathbb P\left(
	\frac1N\sum_{i=1}^NY(Z_i^0)>2M_0
	\right) \le
	\mathbb P\left(
	\left|
	\frac1N\sum_{i=1}^N\widetilde Y_i
	\right|
	>M_0
	\right) \\
	&\le 
	\frac{{\mathbb E} \left[ \left|
	\frac1N\sum_{i=1}^N\widetilde Y_i
	\right|^{p_0} \right]}{M_0^{p_0}} \leq C_{p_0, M_0} N^{-(p_0-1)}.
\end{aligned}
\end{align}
Note that since $p_0 > 1$, the right-hand side of \eqref{D-4-4} tends to zero, as $N \to \infty$. 
\end{proof}
Before we proceed with Step B in the proof of Theorem \ref{T3.2}, we provide the following moment estimates.
\begin{lemma} \label{L5.2}
Let  $f$ be a global weak solution to \eqref{A-2} with initial datum $f_0$ satisfying the following finite conditions:
\[ \int_{\bbr^{2d}} f_0(z) \di z = 1, \quad  \int_{\bbr^{2d}} (  |x|^2  + |v|^2)  f_0(z) \di z < \infty. \]
Then, we have the following estimates:
\begin{align*}
\begin{aligned}
& (i)~ \int_{\bbr^{2d}} |v|^2 f(t, z) \di z \leq \int_{\bbr^{2d}} |v|^2 f_0(z) \di z. \\
& (ii)~ \int_{\bbr^{2d}} |v| f(t, z) \di z \leq \Big( \int_{\bbr^{2d}} |v|^2 f_0(z) \di z \Big)^{\frac{1}{2}}.  \\
& (iii)~ \int_{\bbr^{2d}} |x|^2 f(t,z) \di z \leq 2 \Big(  \int_{\bbr^{2d}} |x|^2 f_0(z) \di z  + t^2   \int_{\bbr^{2d}} |v|^2 f_0(z) \di z       \Big).  \\
& (iv)~ \int_{\bbr^{2d}} |x| f(t,z) \di z \leq \sqrt{2}  \Big[  \Big( \int_{\bbr^{2d}} |x|^2 f_0(z) \di z \Big)^{\frac{1}{2}} + t  \Big(  \int_{\bbr^{2d}} |v|^2 f_0(z) \di z \Big)^{\frac{1}{2}} \Big].
\end{aligned}
\end{align*}
\end{lemma}
\begin{proof} By the conservation of total mass in Proposition \ref{P2.2}, we have
\[    \int_{\bbr^{2d}} f(t, z) \di z = \int_{\bbr^{2d}} f_0(z) \di z = 1, \quad t > 0.  \]
\noindent $\bullet$ (Derivation of (i) and (ii)): The first estimate follows from Lemma \ref{L2.4} (iii). The second estimate follows from the dissipative estimate of total energy in Lemma \ref{L2.4} (iii) using the Cauchy--Schwarz inequality to get 
\[
  \int_{\bbr^{2d}} |v| f(t, z) \di z \leq \Big( \int_{\bbr^{2d}} f(t,z) \di z \Big)^{\frac{1}{2}} \Big( \int_{\bbr^{2d}} |v|^2 f(t, z) \di z \Big)^{\frac{1}{2}} \leq  \Big( \int_{\bbr^{2d}} |v|^2 f_0(z) \di z \Big)^{\frac{1}{2}}.
\]
\noindent $\bullet$ (Derivation of (iii)):~We multiply $|x|^2$ to \eqref{A-2} to find 
\[
\partial_t (|x|^2 f)+ \nabla_x \cdot ( |x|^2 v f) + \nabla_v \cdot ( |x|^2 L[f]f) = 2 x \cdot v f.
\]
We integrate the above relation over $\bbr^{2d}$ using suitable decay at infinity in phase space to get 
\begin{equation}  \label{D-4-5}
\frac{\di }{\di t}\int_{\bbr^{2d}} |x|^2 f(t,z) \di z = 2 \int_{\bbr^{2d}} x \cdot v f(t,z) \di z.
\end{equation}
On the other hand, we use the Cauchy--Schwarz inequality and use Lemma \ref{L2.4} again to find 
\begin{align}
\begin{aligned} \label{D-4-6}
\int_{\bbr^{2d}} |x| |v| f(s,z) \di z &\leq  \Big(  \int_{\bbr^{2d}} |x|^2 f(s, z) \di z \Big)^{\frac{1}{2}} \Big(   \int_{\bbr^{2d}} |v|^2 f(s, z) \di z     \Big)^{\frac{1}{2}}  \\
&\leq \Big(  \int_{\bbr^{2d}} |x|^2 f(s, z) \di z \Big)^{\frac{1}{2}} \Big(   \int_{\bbr^{2d}} |v|^2 f_0(z) \di z \Big)^{\frac{1}{2}}.
\end{aligned}
\end{align}
We combine \eqref{D-4-5} and \eqref{D-4-6} to get 
\begin{align*}
\begin{aligned}
\frac{\di }{\di t}\int_{\bbr^{2d}} |x|^2 f(t,z) \di z \leq  2\Big(  \int_{\bbr^{2d}} |x|^2 f(t, z) \di z \Big)^{\frac{1}{2}} \Big(   \int_{\bbr^{2d}} |v|^2 f_0(z) \di z \Big)^{\frac{1}{2}}.
\end{aligned}
\end{align*}
This yields
\begin{equation} \label{D-4-7}
\frac{\di }{\di t} \Big (\int_{\bbr^{2d}} |x|^2 f(t,z) \di z \Big)^{\frac{1}{2}} \leq \Big(   \int_{\bbr^{2d}} |v|^2 f_0(z) \di z \Big)^{\frac{1}{2}}.
\end{equation}
We integrate \eqref{D-4-7} to get 
\[
 \Big (\int_{\bbr^{2d}} |x|^2 f(t,z) \di z \Big)^{\frac{1}{2}} \leq  \Big (\int_{\bbr^{2d}} |x|^2 f_0(z) \di z \Big)^{\frac{1}{2}} + t \Big(   \int_{\bbr^{2d}} |v|^2 f_0(z) \di z \Big)^{\frac{1}{2}}.
\]
This yields
\begin{align*}
\begin{aligned}
\int_{\bbr^{2d}} |x|^2 f(t,z) \di z &\leq \Big[     \Big (\int_{\bbr^{2d}} |x|^2 f_0(z) \di z \Big)^{\frac{1}{2}} + t \Big(   \int_{\bbr^{2d}} |v|^2 f_0(z) \di z \Big)^{\frac{1}{2}} \Big]^2 \\
&\leq 2 \Big[   \int_{\bbr^{2d}} |x|^2 f_0(z) \di z  + t^2   \int_{\bbr^{2d}} |v|^2 f_0(z) \di z         \Big ].
\end{aligned}
\end{align*}

\noindent $\bullet$ (Derivation of (iv)):~Again, we use the Cauchy--Schwarz inequality to get 
\begin{align*}
\begin{aligned}
\int_{\bbr^{2d}} |x| f(t, z) \di z &\leq \Big( \int_{\bbr^{2d}} |x|^2 f(t, z) \di z \Big)^{\frac{1}{2}} \Big( \int_{\bbr^{2d}}  f(t, z) \di z \Big)^{\frac{1}{2}} = \Big( \int_{\bbr^{2d}} |x|^2 f(t, z) \di z \Big)^{\frac{1}{2}} \\
& \leq \sqrt{2} \Big[  \int_{\bbr^{2d}} |x|^2 f_0(z) \di z  + t^2   \int_{\bbr^{2d}} |v|^2 f_0(z) \di z       \Big]^{\frac{1}{2}} \\
& \leq \sqrt{2}  \Big[  \Big( \int_{\bbr^{2d}} |x|^2 f_0(z) \di z \Big)^{\frac{1}{2}} + t  \Big(  \int_{\bbr^{2d}} |v|^2 f_0(z) \di z \Big)^{\frac{1}{2}} \Big],
\end{aligned}
\end{align*}
where we used $\sqrt{a + b} \leq \sqrt{a} + \sqrt{b},~\forall ~ a, b >0$. 
\end{proof}
With the preparations, we estimate the difference between empirical measure solution and measure-valued solution.
\begin{lemma} \label{L5.3}
For any $T \in (0, \infty)$, we assume that initial data $f_0$ and $Z_0$ satisfy \eqref{B-6}, and let  $f$ and \(\mu^N_t\)  be the weak solution to \eqref{A-2} and the empirical measure generated by  \eqref{A-1} with initial configuration $Z_0$ in the time-interval $[0, T)$, respectively. Then, there exists a random variable $\mathcal R_N$ such that 
\[
\sup_{0\le t \le T}W_1(\mu^N_t, \mu_t)
	\le
	\mathcal R_N  \quad \mbox{and} \quad 	\sup_{N\ge1}\mathbb E\bigl[\mathcal R_N^2\bigr]<\infty.
\]
\end{lemma}
\begin{proof} We split its proof into two steps. \newline

\noindent $\bullet$~\textbf{Step A} (Identification of a candidate for ${\mathcal R}_N$):~For any probability measures \(\mu_1,\mu_2\in\mathcal P_1(\mathbb R^{2d})\),
we have
\begin{equation} \label{D-5-1}
W_1(\mu_1,\mu_2)
\le
\int_{\mathbb R^{2d}} |z| \di \mu_1(z)
+
\int_{\mathbb R^{2d}} |z| \di \mu_2(z).
\end{equation}
Now, we apply \eqref{D-5-1} with \( \mu_1=\mu^N_t \) and \(\mu_2= f(t, z) \di z\) to obtain
\begin{align}
\begin{aligned} \label{D-5-2}
& \sup_{0\le t \le T}W_1(\mu^N_t, \mu_t) \\
& \hspace{0.5cm} \le
	\sup_{0\le t \le T}
	\int_{\mathbb R^{2d}}(|x|+|v|) \di \mu^N_t +
	\sup_{0\le t \le T}
	\int_{\mathbb R^{2d}}(|x|+|v|) f(t,z) \di z =: {\mathcal I}_{31} + {\mathcal I}_{32}.
\end{aligned}
\end{align}
Below, we estimate the terms ${\mathcal I}_{3i},~i=1,2$ one by one. \newline

\noindent $\diamond$~Case 1 (Estimate of ${\mathcal I}_{31}$):~Note that 
\begin{equation} \label{D-5-3}
{\mathcal I}_{31} = \sup_{0\le t \le T}
	\int_{\mathbb R^{2d}}(|x|+|v|) \di \mu^N_t = \sup_{0\le t \le T}
\frac1N\sum_{i=1}^N(|x_i(t)|+|v_i(t)|).
\end{equation}
Then, we use Lemma \ref{L2.1} and \eqref{D-5-3} to get 
\begin{equation} \label{D-5-4}
{\mathcal I}_{31} \le
(1 + T)
\left[
\frac1N\sum_{i=1}^N|x_{i,0}| 
+
\left(\frac1N\sum_{i=1}^N |v_{i,0}|^2\right)^{1/2}
\right].
\end{equation}
\noindent $\diamond$~Case 2 (Estimate of ${\mathcal I}_{32}$):~We use Lemma \ref{L5.2} to get 
\begin{align}
\begin{aligned} \label{D-5-5}
{\mathcal I}_{32} &\leq  \sqrt{2} \Big( \int_{\bbr^{2d}} |x|^2 f_0(z) \di z \Big)^{\frac{1}{2}} +  (1 + \sqrt{2} t)  \Big(  \int_{\bbr^{2d}} |v|^2 f_0(z) \di z \Big)^{\frac{1}{2}} \\
&\leq \sqrt{2} (1 + T) \Bigg[  \Big( \int_{\bbr^{2d}} |x|^2 f_0(z) \di z \Big)^{\frac{1}{2}} +  \Big(  \int_{\bbr^{2d}} |v|^2 f_0(z) \di z \Big)^{\frac{1}{2}}  \Bigg].
\end{aligned}
\end{align}
In \eqref{D-5-2}, we combine estimates \eqref{D-5-4} and \eqref{D-5-5} to get 
\begin{equation} \label{D-5-6}
\sup_{0\le t\le T}W_1(\mu^N_t, \mu_t) \leq C_T \Big[  1 +   \frac1N\sum_{i=1}^N|x_{i,0}| + 
\left(\frac1N\sum_{i=1}^N |v_{i,0}|^2\right)^{1/2} \Big] =: {\mathcal R}_N.
\end{equation}
That is, we have
\begin{equation}\label{D-5-7}
	\sup_{0\le t\le T}W_1(\mu^N_t, \mu_t)
	\le
	\mathcal R_N.
\end{equation}

\vspace{0.2cm}

\noindent $\bullet$~\textbf{Step B} (Second moment estimate for ${\mathcal R}_N$):~Next, we verify that \(\mathcal R_N\) has a uniform second moment. Since
\[
\mathcal R_N
=
C_T
\left[
1+
\frac1N\sum_{i=1}^N|x_{i,0}|
+
\left(\frac1N\sum_{i=1}^N |v_{i,0}|^2\right)^{1/2}
\right],
\]
we use
\[
(a+b+c)^2\le 3(a^2+b^2+c^2),
\qquad \forall ~a, ~b, ~c\ge0,
\]
and linearity of expectation to get
\begin{equation} \label{D-5-8}
	\mathbb E[\mathcal R_N^2]
	\le
	3C_T^2
	\left[
	1+
	\mathbb E\left(\frac1N\sum_{i=1}^N |x_{i,0}|\right)^2
	+
	\mathbb E\left(\frac1N\sum_{i=1}^N |v_{i,0}|^2\right)
	\right].
\end{equation}
Next, we show that the second and third terms in \eqref{D-5-8} are finite. \newline

\noindent $\diamond$~Case 1: We claim that
\begin{equation} \label{D-5-8-1}
\mathbb E\left(\frac1N\sum_{i=1}^N |x_{i,0}|\right)^2 < \infty. 
\end{equation}
{\it (Derivation of \eqref{D-5-8-1})}:~By Jensen's inequality,
\[
\left(\frac1N\sum_{i=1}^N |x_{i,0}|\right)^2
\le
\frac1N\sum_{i=1}^N |x_{i,0}|^2.
\]
Taking expectation and using the fact that \( \{ z_{i,0} \} \) are identically distributed,
we obtain
\begin{equation} \label{D-5-9}
\mathbb E\left(\frac1N\sum_{i=1}^N |x_{i,0}|\right)^2
\le
\frac1N\sum_{i=1}^N\mathbb E|x_{i,0}|^2
=
\int_{\mathbb R^{2d}} |x|^2 f_0(z) \di z.
\end{equation}
Since \(q>2\), we have
\[
|x|^2\le 1+|x|^q.
\]
Hence
\begin{equation} \label{D-5-10}
\int_{\mathbb R^{2d}} |x|^2 f_0(z) \di z
\le
1+\int_{\mathbb R^{2d}} |x|^q f_0(z) \di z <\infty.
\end{equation}
We combine \eqref{D-5-9} and \eqref{D-5-10} to show the desired estimate \eqref{D-5-8-1}.  \newline

\noindent $\diamond$~Case 2: We claim that
\begin{equation} \label{D-5-10-1}
\mathbb E\left(\frac1N\sum_{i=1}^N |v_{i,0}|^2\right) < \infty.
\end{equation}
{\it (Derivation of \eqref{D-5-10-1})}: Note that 
\begin{equation} \label{D-5-11}
\mathbb E\left(\frac1N\sum_{i=1}^N |v_{i,0}|^2\right)
=
\frac1N\sum_{i=1}^N\mathbb E|v_{i,0}|^2
=
\int_{\mathbb R^{2d}} |v|^2 f_0(z) \di z.
\end{equation}
The exponential velocity moment implies the finiteness of the second velocity moment.
Indeed, for all \(r\ge0\),
\[
r^2\le C_\alpha e^{\alpha r}.
\]
Therefore, we have
\begin{equation} \label{D-5-12}
\int_{\mathbb R^{2d}} |v|^2 f_0(z) \di z
\le
C_\alpha
\int_{\mathbb R^{2d}} e^{\alpha |v|} f_0(z) \di z <\infty.
\end{equation}
We combine \eqref{D-5-11} and \eqref{D-5-12} to derive \eqref{D-5-10-1}.  Finally, we combine \eqref{D-5-8}, \eqref{D-5-8-1} and \eqref{D-5-10-1} to find the desired estimate:
\[
	\sup_{N\ge1}\mathbb E\bigl[\mathcal R_N^2\bigr]<\infty.
\]
\end{proof}

\subsection{Proof of Theorem \ref{T3.2}} \label{sec:5.2}
Now, we are ready to provide the proof of the second main result in several steps. Suppose that initial datum $f_0$, $\mu_0^N$, and random initial configuration $Z_0$ satisfy a set of well-prepared conditions \eqref{B-6}, and let  $f$ and \(\mu^N_t\)  be the weak solution to \eqref{A-2} and the empirical measure generated by  \eqref{A-1} with initial configuration $Z_0$. Recall that our goal is to derive an upper bound for \[\displaystyle \mathbb E\left[\sup_{0\le t \le T}W_1(\mu^N_t, \mu_t) \right]\] for a finite-time window which decays to zero as $N \to \infty$. More precisely, we will show that for every finite $T \in (0, \infty)$, there exist positive constants \(C_T\) and \(\eta_T\) such that
	\begin{equation}\label{D-6}
		\mathbb E\left[\sup_{0\le t\le T}W_1(\mu^N_t, \mu_t) \right]\le \frac{C_T}{N^{\eta_T}}.
	\end{equation}
For this,  note that 
\begin{equation}  \label{D-7}
 \mathbb E\left[
	\sup_{0\le t\le T}W_1(\mu^N_t, \mu_t)
	\right] =
	\mathbb E\left[
	\mathbf 1_{\mathcal G_N}
	\sup_{0\le t\le T}W_1(\mu^N_t, \mu_t)
	\right] +
	\mathbb E\left[
	\mathbf 1_{\mathcal B_N}
	\sup_{0\le t\le T}W_1(\mu^N_t, \mu_t)
	\right].
\end{equation}
The first term can be followed from Lemma \ref{L5.1}:
\begin{equation} \label{D-8}
\mathbb E\left[
	\mathbf 1_{\mathcal G_N}
	\sup_{0\le t\le T}W_1(\mu^N_t, \mu_t)
	\right]  \leq C_{T,M }N^{-\eta \Lambda_{T,M}}.
\end{equation}
For the second term in the right-hand side of \eqref{D-7}, we use Lemma \ref{L5.1}, Lemma \ref{L5.3} and the Cauchy--Schwarz inequality to get 
\begin{equation} \label{D-9}
 \mathbb E\left[
	\mathbf 1_{\mathcal B_N}
	\sup_{0\le t\le T}W_1(\mu^N_t, \mu_t)
	\right] \le
	\mathbb E\left[
	\mathbf 1_{\mathcal B_N}\mathcal R_N
	\right]
	\le
	\mathbb P(\mathcal B_N)^{1/2}
	\left(
	\mathbb E[\mathcal R_N^2]
	\right)^{1/2}
	\le
	C_T N^{-\frac{p_0-1}{2}}.
\end{equation}
In \eqref{D-7}, we use estimates \eqref{D-8} and \eqref{D-9} to get 
\begin{align*}
\mathbb E\left[
	\sup_{0\le t\le T}W_1(\mu^N_t, \mu_t)
	\right] \le
	C_TN^{-\eta \Lambda_{T,M_0}}
	+
	C_TN^{-\frac{p_0-1}{2}}.
\end{align*}
Therefore, we set 
\[
\eta_T
:=
\min\left\{
\eta \Lambda_{T,M_0},
\frac{p_0-1}{2}
\right\},
\]
to obtain
\[
\mathbb E\left[
\sup_{0\le t\le T}W_1(\mu^N_t, \mu_t)
\right]
\le
C_TN^{-\eta_T}.
\]
Since
\[
p_0=\min\left\{\frac q2,2\right\},
\]
the constant $\eta_T$ can be rewritten as 
\[
\eta_T
=
\min\left\{
\eta\Lambda_{T,M_0},~
\frac12\left(\min\left\{\frac q2,2\right\}-1\right)
\right\}.
\]
This completes the proof of Theorem \ref{T3.2}.
\begin{remark}\label{R5.1}
In the process of proof, we choose the constant $\eta_T$ as follows:
	\[
	\eta_T
	=
	\min\left\{
	\eta \Lambda_{T,M_0},
	\frac12\left(\min\left\{\frac q2,2\right\}-1\right)
	\right\},
	\]
	where
	\[
	M_0
	:=
	\int_{\mathbb R^{2d}}
	\left(|x|^2+e^{\frac\alpha2 |v|}\right) f_0(z) \di z,
	\quad
	\Lambda_{T,M_0}=e^{-C_{T,2M_0}T}.
	\]
	Equivalently,
	\[
	\eta_T
	=
	\begin{cases}
		\displaystyle
		\min\left\{\eta \Lambda_{T,M_0},\ \dfrac{q-2}{4}\right\},
		& 2<q<4,\\[8pt]
		\displaystyle
		\min\left\{\eta \Lambda_{T,M_0},\ \dfrac12\right\},
		& q\ge4.
	\end{cases}
	\]
	This exponent should be understood as a quantitative rate produced by the present
	argument, not as an optimal rate. The loss in the first term is due to the
	Osgood-type stability modulus in the phase-spatially extended setting, while the second
	term is produced by the good/bad event decomposition and the estimates on
	\(\mathcal B_N\).
\end{remark}
\section{Conclusion}\label{sec:6}
In this paper, we have established finite-time weak stability and random mean-field approximation of the
KCS model in the fully phase-spatially extended setting. Unlike the
classical phase-spatially confined setting, the absence of compact velocity support makes
the alignment force to be non-uniformly Lipschitz along particle trajectories. To
overcome this difficulty, we have introduced a velocity-truncation argument combined
with exponential velocity-tail estimates. These yield a finite-time Osgood-type weak
stability in \(W_1\)-distance for solutions with finite spatial
second moments and exponential velocity moments. As an application of this weak stability, we have also established random
mean-field approximation for empirical measures with non-compact initial
configurations. We also verified an i.i.d. sampling result, showing that random empirical measures generated from random initial data converge to the kinetic
solution on every finite time interval, with an explicit algebraic rate. These
results extend the classical compact-support mean-field theory to a fully
non-compact phase-space framework. Of course, there are several open questions. In particular, it would be interesting to extend
the present method to singular communication weights and to stochastic
nonlinear CS-type models with white noises in the whole space. We leave these interesting questions for future work. 
\section*{Conflict of interest statement}
The authors declare no conflicts of interest.
\section*{Data availability statement}
The data supporting the findings of this study are available from the corresponding author upon reasonable request.
\section*{Ethical statement}
The authors declare that this manuscript is original, has not been published before, and is not currently being considered for publication elsewhere. The study was conducted by the principles of academic integrity and ethical research practices. All sources and contributions from others have been properly acknowledged and cited. The authors confirm that there is no fabrication, falsification, plagiarism, or inappropriate manipulation of data in the manuscript.
\vspace{1cm}

\appendix

\section{Proof of  Lemma \ref{L4.1} } \label{App-A}
\setcounter{equation}{0}

\noindent (ii)~Since \(\phi\in W^{1,\infty}\), we have
\[
\bigl|
\phi(|X_f-X_{f,*}|)
-
\phi(|X_g-X_{g,*}|)
\bigr|
\le
\|\phi'\|_{L^\infty}
\Bigl(
|X_f-X_g|+|X_{f,*}-X_{g,*}|
\Bigr).
\]
This yields
\begin{align}
\begin{aligned} \label{C-7}
	| {\mathcal I}_{12}|
	&\le
	\|\phi'\|_{L^\infty}
	\Bigl(
	|X_f-X_g|+|X_{f,*}-X_{g,*}|
	\Bigr)
	\Bigl(
	|V_g|+|V_{g,*}|
	\Bigr) \\
	&= \|\phi'\|_{L^\infty} \Big( |X_f-X_g|  \cdot  |V_g|  +  |X_f-X_g|  \cdot |V_{g,*}|  \\
	& \hspace{0.4cm} +|X_{f,*}-X_{g,*}| \cdot |V_g| +  |X_{f,*}-X_{g,*}|  \cdot   |V_{g,*}|  \Big) \\
	&=:  \|\phi'\|_{L^\infty} \Big( {\mathcal I}_{121} +{\mathcal I}_{122} +{\mathcal I}_{123} +  {\mathcal I}_{124} \Big).
\end{aligned}
\end{align}
By the exchange symmetry $(z,\bar z)~~\Longleftrightarrow~~(z_*,\bar z_*)$, the estimates for $\iiiint {\mathcal I}_{121}$ and  $\iiiint {\mathcal I}_{124}$ are the same. Similarly, the estimates for $\iiiint {\mathcal I}_{122}$ and  $\iiiint {\mathcal I}_{123}$ are the same. \newline

\noindent $\bullet$ Case A (Estimates for $\iiiint {\mathcal I}_{122}$ and $\iiiint {\mathcal I}_{123}$): By direct calculation, we have
\begin{align}
\begin{aligned} \label{C-7-1}
&  \|\phi'\|_{L^\infty} \iiiint_{\bbr^{8d}} {\mathcal I}_{122} \,\di \pi_0(z,\bar z) \,\di \pi_0(z_*,\bar z_*)  \\
& \hspace{1cm}  =   \|\phi'\|_{L^\infty}  \iiiint_{\bbr^{8d}}  |X_f-X_g|  \cdot |V_{g,*}|  \,\di \pi_0(z,\bar z) \,\di \pi_0(z_*,\bar z_*)   \\
& \hspace{1cm} =  \|\phi'\|_{L^\infty}  \Big( \iint_{\bbr^{4d}}  |V_{g,*}| \di \pi_0(z_*,\bar z_*) \Big) \Big(   \iint_{\bbr^{4d}}  |X_f-X_g| \di \pi_0(z,\bar z) \Big).
\end{aligned}
\end{align}
For the first factor in the right-hand side of \eqref{C-7-1}, we use the Cauchy--Schwarz inequality and the fact that $g_0$ is the ${\bar z}_*$-marginal of the coupling measure $\pi$, we have
\begin{align}
\begin{aligned} \label{C-7-2}
& \iint_{\bbr^{4d}} |V_{g,*}(t,\bar z_*)|\,\di \pi_0(z_*,\bar z_*) \\
& \hspace{1cm} =
\int_{\bbr^{2d}} |V_{g,*}(t,\bar z_*)| g_0(\bar z_*) \di {\bar z}_*
\le
\left(
\int_{\bbr^{2d}} |V_g(t, \bar z_*)|^2 g_0(\bar z_*) \di {\bar z_*}
\right)^{1/2}.
\end{aligned}
\end{align}
It follows from Lemma \ref{L2.2} and Lemma \ref{L2.3} that 
\begin{equation} \label{C-7-3}
\int_{\mathbb R^{2d}} |V_g(t,\bar z_*)|^2 g_0(\bar z_*) \di {\bar z}_*
=
\int_{\mathbb R^{2d}} |v_*|^2 g(t,z_*)\,\di z_*
\le
\int_{\mathbb R^{2d}} |v_*|^2 g_0(z_*)\,\di z_*
\le M.
\end{equation}
On the other hand, the second factor satisfies 
\begin{equation} \label{C-7-4}
  \iint_{\bbr^{4d}}  |X_f-X_g| \di \pi_0(z,\bar z) \leq \Delta(t).
\end{equation}
In \eqref{C-7-1}, we combine all the estimates \eqref{C-7-2}, \eqref{C-7-3} and \eqref{C-7-4} to get 
\begin{equation} \label{C-7-5}
  \|\phi'\|_{L^\infty} \iiiint_{\mathbb R^{8d}} {\mathcal I}_{122} \,\di \pi_0(z,\bar z) \,\di \pi_0(z_*,\bar z_*)  \leq  \|\phi'\|_{L^\infty} \sqrt{M} \Delta.
\end{equation}
Similarly, we have
\begin{equation} \label{C-7-6}
 \|\phi'\|_{L^\infty} \iiiint_{\mathbb R^{8d}} {\mathcal I}_{123} \,\di \pi_0(z,\bar z) \,\di \pi_0(z_*,\bar z_*)  \leq  \|\phi'\|_{L^\infty} \sqrt{M} \Delta.
\end{equation}
\vspace{0.2cm}

\noindent $\bullet$ Case B (Estimates for $\iiiint {\mathcal I}_{121}$ and $\iiiint {\mathcal I}_{124}$):~For $R \geq 1$, note that 
\begin{align}
\begin{aligned} \label{C-7-9}
& \|\phi'\|_{L^\infty} \iiiint_{\mathbb R^{8d}} {\mathcal I}_{121} \,\di \pi_0(z,\bar z) \,\di \pi_0(z_*,\bar z_*) \\
&  \hspace{1cm} =    \|\phi'\|_{L^\infty}  \iiiint_{\mathbb R^{8d}}  |X_f-X_g|  \cdot |V_{g}|  \,\di \pi_0(z,\bar z) \,\di \pi_0(z_*,\bar z_*)  \\
&  \hspace{1cm} =   \|\phi'\|_{L^\infty}  \iint_{\mathbb R^{4d}}  |X_f-X_g|  \cdot |V_{g}|  \,\di \pi_0(z,\bar z) \\
&  \hspace{1cm} =   \|\phi'\|_{L^\infty}  \Big( \iint_{\{|V_g|\le R\}} |V_g|\,|X_f-X_g|\,\di \pi_0(z, \bar z)
	 \\
& \hspace{4cm}	 + 
	\iint_{\{|V_g|>R\}} |V_g|\,|X_f-X_g|\,\di \pi_0(z,\bar z) \Big) \\
& \hspace{1cm}  =: {\mathcal J}_{1211} +  {\mathcal J}_{1212}.
\end{aligned}
\end{align}
In the sequel, we estimate the terms ${\mathcal J}_{121i},~i=1,2$ as follows. \newline

\noindent $\diamond$~Case B.1 (Estimate of ${\mathcal J}_{1211} $):~We use the relation $|V_g|\le R$ over the integral domain to find 
\begin{equation}\label{C-7-10}
	{\mathcal J}_{1211}
	\le  \|\phi'\|_{L^\infty} 
	R\iint_{\mathbb R^{4d}} |X_f-X_g|\,\di \pi_0(z, \bar z)
	\le
	 \|\phi'\|_{L^\infty}  R\Delta.
\end{equation}
\noindent $\diamond$~Case B.2 (Estimate of ${\mathcal J}_{1212}$):~In this case, we use the exponential decay of $|V_g|$ in the integral domain. For this, we use
\[
|X_f-X_g|\le |X_f|+|X_g|
\]
to find 
\begin{align}
\begin{aligned} \label{C-7-11}
	{\mathcal J}_{1212}
	&\le  \|\phi'\|_{L^\infty} 
	\int_{\{|V_g|>R\}} |V_g|\,|X_f|\,\di \pi_0(z, \bar z)
	+  \|\phi'\|_{L^\infty} 
	\int_{\{|V_g|>R\}} |V_g|\,|X_g|\,\di \pi_0(z, \bar z) \\
	& =: {\mathcal J}_{12121} + {\mathcal J}_{12122}.
\end{aligned}
\end{align}
In what follows, we estimate the term ${\mathcal J}_{1212i},~i=1,2$ one by one. We use the Cauchy--Schwarz inequality to find 
\begin{align}\label{C-7-12}
	\begin{aligned}
	 {\mathcal J}_{12121}
		&\le  \|\phi'\|_{L^\infty} 
		\left(
		\int_{\{|V_g|>R\}} |V_g|^2\,\di \pi_0(z,\bar z)
		\right)^{1/2}
		\left(
		\int_{\bbr^{4d}} |X_f|^2\,\di \pi_0(z, \bar z)
		\right)^{1/2},
		\\
		 {\mathcal J}_{12122} 
		&\le  \|\phi'\|_{L^\infty} 
		\left(
		\int_{\{|V_g|>R\}} |V_g|^2\,\di \pi_0(z, \bar z)
		\right)^{1/2}
		\left(
		\int_{\bbr^{4d}} |X_g|^2\,\di \pi_0(z, \bar z)
		\right)^{1/2}.
	\end{aligned}
\end{align}
\noindent $\clubsuit$~Case B.2.1 (Estimate on the first factor in \eqref{C-7-12}):~First, note that for some \(\alpha_1\in(0,\alpha)\), we use Lemma \ref{L2.2} and Lemma \ref{L2.3} to find 
\begin{align}
\begin{aligned} \label{C-7-13}
M &\geq \int_{\bbr^{2d}} e^{\alpha |V_g(t,\bar z)|} g_0(\bar z) \di {\bar z} \geq \int_{\{|V_g|>R \}} e^{\alpha |V_g(t,\bar z)|} g_0(\bar z) \di {\bar z} \\
& =  \int_{\{|V_g|>R \}} e^{(\alpha - \alpha_1) |V_g(t,\bar z)|}  e^{\alpha_1 |V_g(t,\bar z)|}  g_0(\bar z) \di {\bar z} \\
& \geq   e^{(\alpha - \alpha_1) R} \int_{\{|V_g|>R \}} e^{\alpha_1 |V_g(t,\bar z)|}  g_0(\bar z) \di {\bar z},
\end{aligned}
\end{align}
i.e., we have
\begin{equation}\label{C-7-14}
\sup_{0\le t\le T} \int_{\{|V_g|>R \}} e^{\alpha_1 |V_g(t,\bar z)|}  g_0(\bar z) \di {\bar z} \leq M e^{-(\alpha - \alpha_1) R}.
\end{equation}
Then, we use \eqref{C-7-14} and the pointwise bound
\[ |V_g(t,\bar z)|^2\leq  \frac{2}{\alpha_1^2} e^{\alpha_1 |V_g(t,\bar z)|}
\]
to see
\begin{align}
\begin{aligned} \label{C-7-15}
& \int_{ \{ |V_g(t, \bar z)| \geq R \}}   |V_g(t,\bar z)|^2 g_0(\bar z) \di {\bar z} \\
& \hspace{1cm} \leq   \frac{2}{\alpha_1^2} \int_{ \{ |V_g(t, \bar z)| \geq R \}}  e^{\alpha_1 |V_g(t,\bar z)|  }    g_0(\bar z) \di {\bar z} \leq  \frac{2 M}{\alpha_1^2}  e^{-(\alpha - \alpha_1) R}.
 \end{aligned}
 \end{align}
 \noindent $\clubsuit$~Case B.2.2 (Estimate on the second factors in \eqref{C-7-12}):~Note that 
\begin{align}
\begin{aligned} \label{C-7-16}
& \int_{\bbr^{4d}} |X_f(t,z)|^2\,\di \pi_0(z, \bar z) = \int_{\bbr^{2d}} |X_f(t,z)|^2\ f_0(z) \di z, \\
& \int_{\bbr^{4d}} |X_g(t,\bar z)|^2\,\di \pi_0(z, \bar z) =  \int_{\bbr^{2d}} |X_g(t,\bar z)|^2\ g_0(\bar z) \di \bar z.
\end{aligned}
\end{align}
We claim that there exists a positive constant $\tilde{C}$ such that 
\begin{align}
\begin{aligned} \label{C-8}
& \sup_{0\le t\le T} \max \Bigg \{
	\int_{\bbr^{2d}} |X_f(t,z)|^2 f_0(z) \di z, 
	\int_{\bbr^{2d}} |X_g(t,\bar z)|^2\  g_0(\bar z) \di {\bar z}
	\Bigg \}  \\
& \hspace{2cm} \le 2M \Big(1 + \frac{2T^2 }{\alpha^2} \Big)=:\tilde{C}.
\end{aligned}
\end{align}
{\it Proof of \eqref{C-8}}:~By the
characteristic equation \eqref{C-2}, we have
\[
X_f(t,z)=x+\int_0^t V_f(s,z)\,\di s.
\]
This yields
\begin{equation}  \label{C-8-1}
|X_f(t,z)|^2 \leq \Big(|x| +\int_0^t |V_f(s,z)| \di s \Big)^2 \leq 2 |x|^2 +  2\Big| \int_0^t |V_f(s,z)|  \di s \Big|^2.
\end{equation}
To treat the second term in the right-hand side of \eqref{C-8-1}, we use Jensen's inequality to find
\begin{equation} \label{C-8-2}
\Big| \int_0^t |V_f(s,z)|  \di s \Big|^2 \leq t   \int_0^t |V_f(s,z)|^2 \di s.
\end{equation}
We apply \eqref{C-8-1} and \eqref{C-8-2} to get 
\begin{equation} \label{C-8-3}
|X_f(t,z)|^2
\le
2|x|^2+2t\int_0^t |V_f(s,z)|^2\,\di s.
\end{equation}
We integrate \eqref{C-8-3} with respect to \( f_0(z) \di z\) and use Fubini's theorem to obtain
\begin{equation} \label{C-8-4}
\int_{\bbr^{2d}} |X_f(t,z)|^2 f_0(z) \di z
\le
2\int_{\bbr^{2d}} |x|^2\  f_0(z) \di z
+
2t\int_0^t\int_{\bbr^{2d}} |V_f(s,z)|^2 f_0(z) \di z\di s.
\end{equation}
On the other hand, it follows from \eqref{C-1} that 
\[
\frac{\alpha^2 }{2} \int_{\mathbb R^{2d}} |v|^2 f_0(z)\di z \leq \int_{\mathbb R^{2d}}\bigl(|x|^2+e^{\alpha |v|}\bigr) f_0(z)\di z \leq M, 
\]
i.e.,
\begin{equation} \label{C-8-5}
\int_{\mathbb R^{2d}} |v|^2 f_0(z)\di z \leq \frac{2M}{\alpha^2}.
\end{equation}
By Lemma \ref{L2.2}, Lemma \ref{L2.3} and \eqref{C-8-5}, one has 
\begin{equation} \label{C-8-6}
\int_{\mathbb R^{2d}} |V_f(s,z)|^2\  f_0(z) \di z
=
\int_{\mathbb R^{2d}} |v|^2 f(s,z)\,\di z
\le
\int_{\mathbb R^{2d}} |v|^2 f_0(z) \,\di z \leq \frac{2M}{\alpha^2}.
\end{equation}
Now, we combine \eqref{C-8-4} and \eqref{C-8-6} to get 
\begin{align}
\begin{aligned} \label{C-8-7}
\int_{\mathbb R^{2d}} |X_f(t,z)|^2\ f_0(z) \di z &\leq  
2\int_{\mathbb R^{2d}} |x|^2\ f_0(z) \di z
+ \frac{4M t^2 }{\alpha^2} \leq 2M \Big(1 
+ \frac{2T^2 }{\alpha^2} \Big).
\end{aligned}
\end{align}
Similarly, we have
\begin{equation} \label{C-8-8}
\int_{\mathbb R^{2d}} |X_g(t,\bar z)|^2\  g_0(\bar z) \di {\bar z}
\le 2M \Big(1 + \frac{2T^2 }{\alpha^2} \Big).
\end{equation}
Then, we combine \eqref{C-8-7} and \eqref{C-8-8} to get the desired estimate \eqref{C-8}.  

Finally, we collect all the estimates \eqref{C-7-11}, \eqref{C-7-12}, \eqref{C-7-15} and \eqref{C-8} to get 
\begin{equation}\label{C-8-9}
{\mathcal J}_{1212}(t) \le \frac{4M  \|\phi'\|_{L^\infty} }{\alpha_1} \sqrt{ 1 + \frac{2T^2}{\alpha^2}} e^{-\frac{\alpha - \alpha_1}{2} R}.
\end{equation}
Therefore, we have
\begin{align}
\begin{aligned} \label{C-8-99}
 & \|\phi'\|_{L^\infty} \iiiint_{\mathbb R^{8d}} {\mathcal I}_{121} \,\di \pi_0(z,\bar z) \,\di \pi_0(z_*,\bar z_*) \\
 & \hspace{1cm} \leq  \|\phi'\|_{L^\infty}  R\Delta +  \frac{4M  \|\phi'\|_{L^\infty} }{\alpha_1} \sqrt{ 1 + \frac{2T^2}{\alpha^2}} e^{-\frac{\alpha - \alpha_1}{2} R}.
\end{aligned}
\end{align}
Similarly we have
\begin{align}
\begin{aligned} \label{C-8-10}
 & \|\phi'\|_{L^\infty} \iiiint_{\mathbb R^{8d}} {\mathcal I}_{124} \,\di \pi_0(z,\bar z) \,\di \pi_0(z_*,\bar z_*) \\
 & \hspace{1cm} \leq  \|\phi'\|_{L^\infty}  R\Delta +  \frac{4M  \|\phi'\|_{L^\infty} }{\alpha_1} \sqrt{ 1 + \frac{2T^2}{\alpha^2}} e^{-\frac{\alpha - \alpha_1}{2} R}.
\end{aligned}
\end{align}
In \eqref{C-7}, we combine \eqref{C-7-5}, \eqref{C-7-6}, \eqref{C-8-99} and \eqref{C-8-10} to get the desired estimate:
\[
| {\mathcal I}_{12}| \leq  2\|\phi'\|_{L^\infty} \Big( \sqrt{M} \Delta +   R\Delta + \frac{4M}{\alpha_1} \sqrt{ 1 + \frac{2T^2}{\alpha^2}} e^{-\frac{\alpha - \alpha_1}{2} R} \Big).
\]

\end{document}